\documentclass[a4paper,11pt,twoside]{article}

\usepackage{amsmath, amssymb, amsthm}
\usepackage{color}

\usepackage{graphicx}

\newtheorem{thm}{Theorem}
\newtheorem{prop}{Proposition}
\newtheorem{lemma}{Lemma}
\newtheorem{cor}{Corollary}
\newtheorem{remark}{Remark}
\newtheorem{example}{Example}

\numberwithin{equation}{section}
\allowdisplaybreaks

\def\dis{\displaystyle}
\def\R{\mathbb{R}}
\def\Z{\mathbb{Z}}

\def\N{\mathbb{N}}
\def\H{\mathbb{H}}
\def\Y{\mathbb{Y}}

\def\e{{\epsilon}}
\def\D{\Delta}
\def\d{\delta}
\def\l{\left}
\def\r{\right}
\def\a{\alpha}
\def\b{\beta}
\def\G{\Gamma}
\def\g{\gamma}
\def\th{\theta}
\def\s{\sigma}
\def\z{\zeta}
\def\k{\kappa}
\def\la{\lambda}

\def\n{\nabla}
\def\p{\partial}
\def\F{\mathcal{F}}

\def\P{{\mathbb{P}}}
\def\toP{\stackrel{\P}{\longrightarrow}}
\def\toD{\stackrel{\mathcal{D}}{\longrightarrow}}
\def\E{\mathbb{E}}

\def\mb#1{\mbox{\boldmath $#1$}}

\def\I{\mathbf{1}}

\def\wh#1{\widehat{#1}} 
\def\ol#1{\overline{#1}} 
\def\wt#1{\widetilde{#1}} 
\def\df{{\rm d}}
\def\DX{\D_k^nX}

\def\t{{t_{k-1}^n}}
\def\T{{t_k^n}}
\def\filter{\I_{\{|\DX|\le \d_{n,\e}\}}} 
\def\qfilter{\I_{\l\{\|\D Q^\e\|^*_k \le \d\r\}}} 

\title{Threshold estimation for stochastic processes with small noise}
\author{Yasutaka Shimizu\footnote{E-mail: shimizu@waseda.jp } \\ 
{\it Department of Applied Mathematics, Waseda University}}
\date{March 16, 2017}

\begin{document}

\maketitle 

\begin{abstract} 

Consider a process satisfying a stochastic differential equation with unknown drift parameter, and suppose that discrete observations are given. It is known that a simple least squares estimator (LSE) can be consistent, but numerically unstable in the sense of large standard deviations under finite samples when the noise process has jumps. We propose a filter to cut large shocks from data, and construct the same LSE from data selected by the filter. The proposed estimator can be asymptotically equivalent to the usual LSE, whose asymptotic distribution strongly depends on the noise process. However, in numerical study, it looked asymptotically normal in an example where filter was choosen suitably, and the noise was a L\'evy process. 
We will try to justify this phenomenon mathematically, under certain restricted assumptions.

\begin{flushleft}
 \vspace{1mm}
\end{flushleft}
\end{abstract} 

\section{Introduction}\label{sec:intro}

Let $(\Omega,\F,(\F_t)_{t\ge 0},\P)$ be a stochastic basis, on which an $\R^d$-valued stochastic process $X$ is 
defined via the stochastic integral equation
\begin{align}
X_t^\e = x + \int_0^t b(X_s^\e,\th_0)\,\df s + \e\cdot Q_t^\e,  \label{model}
\end{align}
where $x\in \R^d,\e > 0$, and $\th_0$ is an unknown parameter that belongs to a parameter space $\Theta_0$, which is an open bounded, convex subset of $\R^p$; 
we put $\Theta :=\ol{\Theta_0}$, the closure of $\Theta_0$; $b$ is a measurable function on $\R^d\times\Theta$; 
$Q^\e$ is a stochastic process of the form 
\[
Q^\e_t := \wt{Q}^\e_t + r_\e B^H_t
\]
where $\wt{Q}^\e$ is a semimartingale, $B^H$ is a fractional Brownian motion with the Hurst parameter $H\in (0,1)$, and 
$r_\e$ is a real number satisfying that $r_\e\to 0\,(\e\to 0)$. 
We also assume for the semimartingale  $\wt{Q}^\e$ that $\wt{Q}^\e_0=0$ and $\sup_{t\in [0,1]}\l|\wt{Q}^\e_t - Q_t\r|\to 0$ as $\e\to 0$ with probability one. 
Hence we have that 
\begin{align}
\sup_{t\in [0,1]}\l|Q^\e_t - Q_t\r| \to 0\quad a.s., \quad \e\to 0, \label{Q}
\end{align}

As is well known, $B^H$ is not a semimartingale and neither is $Q^\e$, but $Q^\e$ converges uniformly to a semimartingale $Q$ as $\e\to 0$ almost surely. 

Moreover, we suppose that the Doob-Meyer decomposition of $Q$ is given as follows. 
\begin{align}
Q_t = A_t + M_t\quad (t\ge 0),\quad A_0=M_0=0\quad a.s., \label{doob}
\end{align}
where $A$ is a process with finite variation, and $M$ is an $\F_t$-local martingale. 
Some typical examples for $Q^\e$ is given in Section \ref{sec:examples}.
We suppose that the process $X^\e=(X^\e_t)_{t\in [0,1]}$ is observed discretely in time: $\{X_\T\}_{k=0}^n$ with $\T=k/n$ (the index $\e$ is omitted in the notation), 
from $[0,1]$-interval. We denote by $\DX:=X_{t_k^n} -X_\t$ and $\D_n:=\T-\t=1/n$. 
Our interest is to estimate the value of the parameter $\th_0$ from the discrete samples under that  $n\to \infty$ as well as $\e\to 0$: {\it small noise asymptotics}. 

There are some practical advantages in small noise asymptotics: 
\begin{itemize}
\item {\it Statistical point of view}: 
the drift estimation by samples from a fixed finite time interval is justifiable under relatively mild conditions. Since we need to observe the process long time to achieve ``good" estimation of the drift without small noise assumption, 
we usually assume the asymptotics that the terminal time of observations goes to infinity, under which some technical conditions such as ``ergodicity" or ``uniform moment conditions" for the process need to be assumed. By a suitable scaling technique, we can regard such a long-term model as a small noise model, approximately. 
Then there is no need for those conditions, which are sometimes difficult to check in practice; 
see remarks in Section \ref{sec:long-term} for details.
 
\item {\it Computational point of view}: approximation of functionals of the process as $\e\to 0$ is often available in relatively easy-to-calculate form as in, e.g., 
Yoshida \cite{y92a} and Pavlyukevich \cite{pav11} among others, which is well applied to finance and insurance; 
see Takahashi \cite{t99}, Kunitomo and Takahashi \cite{kt01}, Takahashi and Yoshida \cite{ty04}, Uchida and Yoshida \cite{uy04b}, Pavlyukevich \cite{pav08} and references therein. 
\end{itemize}
That is why, using the small noise model is convenient to deal with both applications and statistical inference at the same time. 

Sampling problems for stochastic differential equations with small noise have been well studied by many authors in both theoretical and applied point of views. 
Some earlier works for {\it small-diffusion models} are found in the papers by Kutoyants \cite{k84,k94}, Genon-Catalot \cite{gc90} and Laredo \cite{la90}, and 
they have been developed in some directions by several authors: e.g., martingale estimating functions are studied by S{\o}rensen \cite{so00}, 
efficient estimation is investigated by S{\o}rensen and Uchida \cite{su03}, Gloter and S{\o}rensen \cite{gs09}; see also Uchida \cite{u04,u08},  
and asymptotic expansion approach is initiated by Yoshida \cite{y92a,y92b}, see also Uchida and Yoshida \cite{uy04b}, among others. 
Although those works are due to diffusion noise, more general noise model are also considered recently. 
For example, Long \cite{l09} and Ma \cite{m10} investigate the drift estimation of a L\'evy driven Ornstein-Uhlenbeck process; see also Long \cite{l10}, 
and Long {\it et al.} \cite{lss13} deal with the inference for non-linear drift under the small semimartingale noise. 
In our paper, we do not require that the noise is a semimartingale, but `approximately' a semimartingale in the sense of \eqref{Q}. 

The goal of this paper is statistical inference for the drift of the process under the asymptotics that noise vanishes, which is the same motivation as in Long {\it et al.} \cite{lss13}, 
but it was seen in our numerical study that our estimator performed better than theirs under finite samples. 
Long {\it et al.} \cite{lss13} consider the case where $Q^\e \equiv L$ is a L\'evy process (although it can be extended to a case of a semimartingale) 
independent of the dispersion parameter $\e$, 
and investigate the asymptotic behavior of the {\it least squares-type estimator (LSE)} defined by 
\begin{align}
\wh{\th}_{n,\e}^{LSE} := \arg\min_{\th\in\Theta} \Psi_{n,\e}(\th), \label{lse1}
\end{align}
where 
\begin{align}
\Psi_{n,\e}(\th)=\e^{-2}\D_n^{-1}\sum_{k=1}^n |\DX - b(X_\t,\th)\cdot \D_n|^2\label{lse2}
\end{align}
They show that the minimum contrast estimator $\wh{\th}_{n,\e}^{LSE}$ is $\e^{-1}$-consistent with the limit of a L\'evy functional as $\e\to 0$ 
and $n\to \infty$ with $n\e\to \infty$ under some mild conditions on the function $b$; see Theorems 4.1 and 4.2 in \cite{lss13}. 
The asymptotic distribution generally has a fat-tail, that causes us unsatisfactory performance even if $\e$ is small enough; 
see numerical results in \cite{lss13}, or Section \ref{sec:simulation} below. 
This would be due to `large' shocks by the driving noise. 
It will be easy to imagine that a `large' jump of $Q^\e$ makes much impact to the direction of drift, and make the drift estimation unstable. 
Therefore, cutting such `large' jumps could improve the performance. 
That is why, we consider the {\it threshold-type} estimator defined as follows: 
\begin{align}
\wh{\th}_{n,\e} := \arg\min_{\th\in\Theta} \Phi_{n,\e}(\th), \label{tte1}
\end{align}
where 
\begin{align}
\Phi_{n,\e}(\th)=\e^{-2}\D_n^{-1}\sum_{k=1}^n |\DX - b(X_\t,\th)\cdot \D_n|^2\I_{\{|\DX|\le \d_{n,\e}\}}, \label{tte2}
\end{align}
$\d_{n,\e}$ is a positive number, which is a threshold to eliminate `large' shocks causing bias to drift estimation. 
That is, the indicator $\I_{\{|\DX|\le \d_{n,\e}\}}$ plays a role of a filter to split increments with `large' and `small' magnitude of shocks; 
see Shimizu \cite{s06}, or Shimizu and Yoshida \cite{sy06} for the fundamental idea of those filters. 
It would be intuitively clear that if $\d_{n,\e}\to \infty$ then $\wh{\th}_{n,\e}^{LSE}$ and $\wh{\th}_{n,\e}$ can be asymptotically equivalent. 
However, we will show that the same thing holds true even if $\d_{n,\e}\to 0$ by choosing the sequence $\d_{n,\e}$ carefully. 
Moreover we can see in numerical study that $\wh{\th}_{n,\e}$ has much better finite-sample performance than that of $\wh{\th}_{n,\e}^{LSE}$. 
Furthermore, we will show the {\it mighty convergence} (the convergence of moments) for $\wh{u}_{n,\e}:=\e^{-1}(\wh{\th}_{n,\e} - \th_0)$, 
which is a stronger result than those in \cite{lss13}: 
\[
\E\l[f(\wh{u}_{n,\e})\r] \to \int_{\R^p} f(z)\,{\cal L}_u(\df z),
\]
as $\e\to 0$ and $n\to \infty$ for every continuous function $f$ of at most polynomial growth, 
where ${\cal L}_u$ is the asymptotic distribution of $\wh{u}_{n,\e}$. 

As is described above, the asymptotic distribution ${\cal L}_u$ is generally not normal unless the limiting process $Q$ is a Wiener process. 
However, it is interesting to note that we sometimes encounter a phenomenon that $\wh{u}_{n,\e}$ seems asymptotically normal 
in numerical study. 
This may indicate that our filtered LSE could also be asymptotically normal if we choose $\d_{n,\e}$ in a different way as previous, 
otherwise we may just observe it as an `approximate' phenomenon possibly when some appropriate conditions are satisfied. 
Although a justification of this phenomena under the discrete sampling is still an open problem, 
we will add some discussion on these points in Section \ref{sec:normal?}.

The paper is organized as follows. In Section \ref{sec:main}, we prepare notation and assumptions, and present the main results under discrete samples. 
In particular, Section \ref{sec:Q} is devoted to investigating some technical conditions for $Q$. We will give some easy-to-check sufficient conditions for those 
when the noise $Q$ is a L\'evy process. 
In Section \ref{sec:simulation}, we will show an advantage of our estimator compared with the usual LSE via numerical study, and we further observe 
the asymptotic distribution seems normal. Finally, we will try to give a theoretical explanation to those asymptotic phenomena in Section \ref{sec:normal?}. 
All the proofs of main theorems are given in Section \ref{sec:proof}.

\section{Main results}\label{sec:main}

\subsection{Notation and assumptions}\label{sec:notation}

We use the following notation: 
\begin{itemize}
\item For a process $Y=(Y_t)_{t\in [0,1]}$, $\|Y\|_{L^p} = (\E|Y_1|^p)^{1/p}$ $(p>0)$ and $\|Y\|_*:=\sup_{t\in [0,1]}|Y_t|$. 


\item Given a multilinear form ${\cal M}=\{M^{(i_1,\dots,i_K)}\,:\,i_k=1,\dots,d_k; k=1,\dots,K\} \in \R^{d_1}\otimes \dots \otimes \R^{d_K}$ and vector $u_k=(u_k^{(i)})_{i\le d_k} \in \R^{d_k}$, 
we write 
\[
{\cal M}[u_1,\dots,u_K] = \sum_{i_1=1}^{d_1}\dots\sum_{i_K=1}^{d_K}M^{(i_1,\dots,i_K)} u_1^{(i_1)} \dots u_K^{(i_K)}. 
\]
Note that the above form is well-defined when the $j$th dimension of ${\cal M}$ (the number of $i_j$) and that of $u_j$ are the same. 
When some of $u_k$ is missing in ``${\cal M}[u_1,\dots,u_K]$", the resulting form is regarded as a multilinear form again; e.g., ${\cal M}[u_3,\dots,u_K] \in \R^{d_1}\otimes \R^{d_2}$. 
For example, when ${\cal M}$ is a vector 
$M=(M_i)_{1\le i\le d}$, ${\cal M}[x] = M^\top x$ for $x\in \R^d$, which is the inner product, and when $\wt{{\cal M}}$ is a matrix $M=(M_{ij})_{1\le i\le d_1; 1\le j\le d_2}$, 
$\wt{{\cal M}}[u,v]$ for $u=(u_i)_{1\le i\le d_1}$ and $v=(v_i)_{1\le i\le d_2}$ is the quadratic form $u^\top M v$, and 
$\wt{{\cal M}}[v]=M v\in \R^{d_1}$ or $\wt{{\cal M}}[u] = M^\top u\in \R^{d_2}$, among others. 
The correspondences of dimensions will be clear from the context. We also use the notation ${\cal M}[u^{\otimes K}]:={\cal M}[u_1,\dots,u_K]$ when $u_1=\dots=u_K$. 

\item For $a=(a_1,\dots,a_m)^\top\in \R^m$, $\n_a=(\p/\p{a_1},\dots,\p/\p{a_m})$. 
Moreover, we denote by $\n_a^k:=\n_a\otimes \n_a^{k-1}$ with $\n_a^0\equiv 1$. That is, $\n_a^k$ forms a multilinear form, e.g., $\n_a^2 = \n_a^\top \n_a$ in a matrix form.

\item For a multilinear form ${\cal M}$, $|{\cal M}|^2$ denotes the sum of the squares of each element of ${\cal M}$.  

\item $C$ is often used as a generic positive constant that may differ from line to line. 
Moreover, we write $a \lesssim b$ if $a \le Cb$ almost surely.  

\item $C^{k,l} (\R^d\times\Theta;\R^q)$ denotes the space of functions $f(x,\th):\R^d\times\Theta\to \R^q$ that is $k$ and $l$ times differentiable with respect to $x$ and $\th$, 
respectively. Moreover, $C^{k,l}_\uparrow (\R^d\times\Theta;\R^q)$ denotes a subclass of $f \in C^{k,l}(\R^d\times\Theta;\R^q)$ that is of polynomial growth uniformly in $\th\in\Theta$: 
$\sup_{\th\in\Theta}|\n_x^\a\n_\th^\b f(x,\th)| \lesssim (1 + |x|)^C$ for any $\a\le k$ and $\b\le l$. 

\item For a function $g(x,\th):\R^d\times \Theta \to \R^q$, we write $g_{k-1}(\th):=g(X_\t,\th)$. Moreover, denote by $\chi_k(\th):=\DX - b_{k-1}(\th)\D_n (\in \R^d)$. 

\item All the asymptotic symbol are described under $n\to \infty$ and $\e\to 0$ unless otherwise noted. 

\end{itemize}

Using the above notation, our estimating function given in \eqref{tte2} is rewritten as 
\begin{align*}
\Phi_{n,\e}(\th) 
&= \e^{-2}\D_n^{-1}\sum_{k=1}^n I[\chi_{k}^{\otimes 2}(\th)]\I_{\{|\DX|\le \d_{n,\e}\}},  
\end{align*}
where $I$ is the $d\times d$ identity matrix.  

We make the following assumptions on the model \eqref{model}: 
\begin{description}
\item[{\bf A1}] $|b(x,\th) - b(y,\th)|\lesssim |x-y|$ for each $x,y\in \R^d$ and $\th\in\Theta$. 
\end{description}
Under this assumption, the ordinary differential equation
\[
\df X^0_t = b(X_t^0,\th_0)\,\df t,\quad X_0^0 = x, 
\]
has the unique solution $X^0=(X_t^0)_{t\ge 0}$. 
\begin{description}
\item[{\bf A2}] $b\in C^{2,3}_\uparrow(\R^d\times\Theta;\R^d)$. 
\item[{\bf A3}] $\th \ne \th_0$ $\Leftrightarrow$ $b(X_t^0,\th)\ne b(X_t^0,\th_0)$ for at least one value of $t\in [0,1]$. 
\item[{\bf A4}] $I(\th_0):=\int_0^1 \n_\th b(X_t^0,\th_0)^\top \n_\th b(X_t^0,\th_0)\,\df t\ (\in \R^p \otimes \R^p)$ is positive definite, 
\end{description}

We further make the following conditions for the limiting process $Q$ of $Q^\e$. 
\begin{description}
\item[{\bf Q1[$\mb{\g}$]}] There exists some $\g>0$ such that, for any $k=1,\dots,n$, 
\[
\P\l\{\sup_{t\in (\t,\T]}|Q_t- Q_\t| > \D_n^\g\Big|\F_\t\r\}=o_p(1).
\] 
\item[{\bf Q2[$\mb{q}$]}] For $q>0$ and processes $A$ and $M$ given in \eqref{doob}, the total variation of $A$, say $TV(A):=\int_0^1 |\df A_t|$, 
and the {\it quadratic variation} $[M,M]$ satisfy that
\[
\E\l[TV(A)^q\r] + \E\l[[M,M]_1^{q/2}\r] <\infty. 
\] 
\end{description}

Although the condition Q1 seems {\it ad hoc}, we can give some easy-to-check conditions in some important cases where, e.g., 
$Q$ is a L\'evy process satisfying Q2[$q$] for some $q>0$; see Section \ref{sec:Q} for details.

\subsection{Asymptotic behavior of threshold-type estimators}\label{sec:theorem}
\begin{thm}\label{thm:consist} Suppose A1--A3, Q1[$\g$], and that a sequence $\{\d_{n,\e}\}$ satisfies that 
\begin{align}
\d_{n,\e}\D_n^{-1}\to \infty,\quad \e\,\D_n^{\g}\,\d_{n,\e}^{-1} =O(1),\quad n\e\to \infty.  \label{d-1}
\end{align}
Then 
\[
\wh{\th}_{n,\e}\toP \th_0. 
\]
\end{thm}

\begin{remark}
Condition \eqref{d-1} ensures a kind of ``negligibility": 
\begin{align*}
\P\l(|\DX|> \d_{n,\e} |\F_\t\r) \to 0, 
\end{align*}
which makes $\wh{\th}_{n,\e}$ asymptotically equivalent to $\wh{\th}_{n,\e}^{LSE}$; see \eqref{Q1} in the proof. 
\end{remark}

\begin{thm}\label{thm:asym-dist} Suppose the same assumptions as in Theorem \ref{thm:consist}, and further A4. 
Then 
\[
\e^{-1}(\wh{\th}_{n,\e} - \th_0) \toP \z:=I^{-1}(\th_0)\int_0^1 \n_\th b(X_t^0,\th_0)[\df Q_t],  
\]
where the square bracket $[\df Q_t]$ is the quadratic form of $\df Q_t$ for the matrix $\n_\th b(X_t^0,\th_0)$ in this case; see Section \ref{sec:notation}. 
\end{thm}

\begin{thm}\label{thm:mighty} Suppose the same assumptions as in Theorem \ref{thm:asym-dist}, and that Q2[$q$] holds true for any $q>0$. 
Then, it follows for every continuous function $f:\R^p\to \R$, of polynomial growth that 
\[
\E\l[f\l(\e^{-1}(\wh{\th}_{n,\e} - \th_0)\r)\r] \to \E[f(\z)], 
\]
where $\z$ is given in Theorem \ref{thm:asym-dist}. 
\end{thm}

\subsection{Examples and remarks}\label{sec:examples}

\subsubsection{$\a$-stable noise}
Suppose $d=1$, and that $Q^\e =\s B + \rho_\e S$, where $\rho_\e\to \rho\in [0,\infty)$, $\s\ge 0$ is a constant, $B$ is a Brownian motion, 
and $S$ is a standard $\a$-stable process with stability index $\a\in (0,2)$ and the skewness parameter $\b\in [-1,1]$, 
which is denoted by $S_\a(1,\b,0)$; see, e.g., Cont and Tankov \cite{ct04}, page 94 for this notation. 
In this case, the limiting variable $\z$ in Theorem \ref{thm:asym-dist} becomes 
\begin{align*}
\z=I^{-1}(\th_0) \l\{ \s D_2\cdot Z + \rho D_\a\cdot S_\a(1,\b,0)\r\}, 
\end{align*}
where $D_\a = \l(\int_0^1 \{\p_\th b(X_t^0,\th_0)\}^\a \,\df t\r)^{1/\a}$ and $Z$ is a standard Gaussian variable. 
Hence if $\rho= 0$ then the estimator is asymptotically normal; see also Long {\it et al.} \cite{lss13}. 

\subsubsection{Markovian noise}
Let us consider the case where 
\begin{align}
\df Q_t^\e = \e_1 a(X_t^\e)\,\df W_t +  \e_2 c(X_{t-}^\e)\,\df Z_t, \label{Levy driven sde}
\end{align}
where $\e_j> 0\ (j=1,2)$, $a,c$ are some suitable functions, and $Z$ is a pure jump L\'evy process. 
\begin{itemize}
\item Case of $\rho_\e:=\e_2/\e_1 \to 0$: we can reparametrize as $\wt{\e}:=\e\cdot\e_1$ to obtain that 
\[
\df X_t^\e = b(X_t^\e,\th_0)\,\df t + \wt{\e}\cdot \df \wt{Q}^\e_t, 
\]
with 
\[
\df \wt{Q}^\e_t = a(X_t^\e)\,\df W_t + \rho_\e c(X_{t-}^\e)\,\df Z_t. 
\]
Then $\wt{\e}^{-1}(\wh{\th}_{n,\e} - \th_0)$ is asymptotically normal as in S{\o}rensen and Uchida \cite{su03}. 
We can verify that 
\[
\wt{Q}^\e_t \to Q:=\int_0^t a(X_s^0)\,\df W_t\quad a.s., \quad \e\to 0, 
\]
uniformly in $t\in [0,1]$ from the fact that 
the stochastic integrals are continuous with respect to the {\it u.c.p. topology}; see Theorem II.11 by Protter \cite{p04}, 
and the uniform convergence of $X^\e \to X^0$ on compact sets; see Theorem IX.4.21 by Jacod and Shiryaev \cite{js03}. 

\item Case of $\rho_\e:=\e_2/\e_1\to \rho \in (0,\infty)$: we can use the same model as above, 
and the asymptotic distribution of $\wt{\e}^{-1}(\wh{\th}_{n,\e} - \th_0)$ is a convolution of stochastic integrals with respect to $W$ and $Z$.

\item Case of $\rho_\e:=\e_2/\e_1\to\infty$: we can reparametrize as $\ol{\e}:=\e\cdot\e_2$ to obtain that 
\[
\df X_t^\e = b(X_t^\e,\th_0)\,\df t + \ol{\e}\cdot \df \ol{Q}^\e_t, 
\]
with 
\[
\df \ol{Q}^\e_t = \rho_\e^{-1} a(X_t^\e)\,\df W_t + c(X_{t-}^\e)\,\df Z_t. 
\]
Then the asymptotic distribution of $\ol{\e}^{-1}(\wh{\th}_{n,\e} - \th_0)$ is written by a stochastic integral with respect to $Z$ only.
\end{itemize}

\subsubsection{Long-term observations}\label{sec:long-term}

Consider a `long-term' discretely observed model: 
\begin{align}
X_t =x +  \int_0^t b(X_v)\,\df v + Q_t, \label{original}
\end{align}
and $X$ is observed at $0=t_0^n<t^1_n<,\dots,<t_n^n=:T$ for a fixed $T>0$ `large enough', 
where $Q$ is an $\a$-stable process: $Q_1\sim S_\a(\s,\b,0)$ with index $\a\in (1,2)$. 
To estimate the drift function $b$, we sometimes assume a parametric model $b(x)=b(x,\th)$ and estimate $\th$ 
under the assumption that $T\to \infty$. 
However, in a standard theory of parametric inference, we usually need an ``ergodicity" or uniform moment conditions such as $\sup_{t>0}\E|X_t|^m<\infty$ for any $m>0$, 
which are often restrictive conditions.  

Now, transform the model by $v=T u$, $t=Ts$, and divide the both sides of \eqref{original} by $T$ to obtain 
\[
Y_{s} = \frac{x}{T} +  \int_0^{s} b(T\cdot Y_u,\th)\,\df u + \frac{1}{T}Q_{Ts},\quad s\in [0,1], 
\]
where $Y_s = T^{-1}X_{T s}$. Since $Q_{Ts} =^d T^{1/\a}Q_s$ by the {\it self-similarity} of stable processes, we can regard that $Y$ is a (weak) solution to the following SDE: 
\[
Y_{s} = \frac{x}{T} +  \int_0^{s} b(T\cdot Y_u,\th)\,\df u + T^{1/\a-1}\wt{Q}_{s}, 
\]
where $\wt{Q}_{s}$ is also an $\a$-stable process such that $\wt{Q}_1\sim S_\a(\s,\b,0)$. 
Suppose that $\e:=T^{1/\a-1}$ is small enough for given $T$, and putting $\wt{b}(x,\th)=b(T x,\th)$ and $\wt{x}=\e^{\a/(\a-1)}x$, 
we can reformulate the model as
\begin{align}
Y_s = \wt{x} + \int_0^s \wt{b}(Y_u,\th)\,\df u + \e \cdot \wt{Q}_s, \quad s\in [0,1], \label{transform}
\end{align}
with $\e>0$ small enough, and it is interpreted as a small noise SDE.

Given a ``non-ergodic" model as in \eqref{original} with long-term observations, 
we can reformulate it to a small noise model \eqref{transform} with known constant $\e=T^{1/\a-1}$ by plotting the data $\{Y_{s_k^n}:=T^{-1}X_{Ts_k^n}\}_{k=0}^n$ with $s_k^n=\T/T$ and $\wt{x}=Y_0$ in the interval $[0,1]$. Then, although $\wt{b}$ depends on $T$, 
we can formally use a estimator of $\th$ in $\wt{b}$ under the small noise model that does not require any restrictive condition.

\subsection{On conditions for $Q$}\label{sec:Q}

Let us investigate some sufficient conditions to ensure Q1 and Q2 when $Q$ is specified. 

An important case is when $Q$ is a L\'evy process such that the {\it characteristic exponent}: $\psi(u) := \log \E\l[\exp(i u^\top Q_1)\r]$, is given by 
\begin{align}
\psi(u) = ib^\top u - \frac{\s^2}{2}u^\top u + \int_{\R^d}\l(e^{iu^\top z} - 1- \frac{iu^\top z}{1+|z|^2}\r)\,\nu(\df z),\quad u\in \R^d, \label{Q:levy} 
\end{align}
where $b\in \R^d$, $\nu$ is the L\'evy measure with $\nu(\{0\})=0$ and $\int_{|z|\le 1}|z|^2\,\nu(\df z) < \infty$. 

As a simple case such that $Q$ is a Wiener process, where $b=0,\nu\equiv 0$, 
it follows from the property of the stationary, independent increments that 
\begin{align*}
\P\l\{\sup_{t\in (\t,\T]}|W_t-W_\t| > \D_n^\g\Big|\F_\t\r\} 
&=\P\l\{\sup_{t\in (0,\D_n]}|W_t| > \D_n^\g\r\} \\
&= 2\l(1-\Phi(\D_n^\g/\sqrt{\D_n})\r) \to 0,  
\end{align*}
for any $\g\in (0,1/2)$, where $\Phi$ is a standard normal distribution function; for the last equality, see, e.g., Doob \cite{d49}, or Boukai \cite{b90}.  

When $Q$ is a stable process, we have the following result. 
\begin{prop}\label{prop:stable}
When $Q$ is a symmetric $\a$-stable process with $\a\in (1,2)$. Then Q1[$\g$] holds true for any $\g\in (0,\a^{-1})$. 
\end{prop}
\begin{proof}
Due to the maximal inequality (3.5) in Joulin \cite{j07}, we have 
\begin{align*}
\P\l\{\sup_{t\in (\t,\T]}|Q_t - Q_\t| > \D_n^\g\Big|\F_\t\r\} 
&=\P\l\{\sup_{t\in (0,\D_n]}|Q_t| > \D_n^\g\r\} =O\l( \D_n^{1 - \g\a}\r)\to 0
\end{align*}
as $n\to \infty$. 
\end{proof}

We can also present some sufficient conditions to Q1[$\g$] when a process $Q$ is a more general L\'evy process: let 
\begin{align*}
h(x) &:= \int_{|z|>x}\nu(\df z) + x^{-2}\int_{|z|\le x} |z|^2\,\nu(\df z)  \\
&\quad + x^{-1}\l|b + \int_{|z|\le x} \frac{z|z|^2}{1 + |z|^2}\,\nu(\df z) - \int_{|z|> x} \frac{z}{1 + |z|^2}\,\nu(\df z)\r|. 
\end{align*}

\begin{prop}\label{prop:1}
Suppose that $Q$ is a L\'evy process with characteristic \eqref{Q:levy}, and that there exists a constant 
\begin{align}
\b:= \inf\l\{\eta>0:\limsup_{x\to 0} x^\eta h(x) = 0\r\}. \label{beta}
\end{align}
Then the condition Q1[$\g$] holds true for any $\g \in (0,\g_0)$, where 
\[
\g_0=\l\{ \begin{array}{ll} \b^{-1}& (\s=0)\\ (\b\vee 2)^{-1}& (\s\ne 0)\end{array}\r.. 
\] 
We interpret that $1/0=\infty$. 
\end{prop}

\begin{proof}
Let $Q_t=\s W_t + Z_t$, where $W$ is a Wiener process and $Z$ is a pure jump L\'evy process with characteristic 
\[
\log \E\l[\exp(i u^\top Z_1)\r]=ib^\top u + \int_{\R^d}\l(e^{iu^\top z} - 1- \frac{iu^\top z}{1+|z|^2}\r)\,\nu(\df z),\quad u\in \R^d.
\]
By the independent, stationary increments property for L\'evy process $Q$, we see that 
\begin{align*}
R_n&:= \P\l\{\sup_{t\in (\t,\T]}|Q_t-Q_\t| > \D_n^\g\Big|\F_\t\r\} \\
&\lesssim \P\l\{\sup_{t\in (0,\D_n]}|\s W_t| > \D_n^\g/2\r\} + \P\l\{\sup_{t\in (0,\D_n]}|Z_t| > \D_n^\g/2\r\} =: R^{(1)}_n + R^{(2)}_n. 
\end{align*}
Note that $R^{(1)}_n\to 0$ for any $\g\in (0,1/2)$ if $\s\ne 0$, and that $R^{(1)}_n\equiv 0$ if $\s=0$. 
Moreover, according to Pruitt \cite{p81}, (3.2), it follows that 
\[
R^{(2)}_t \lesssim \D_n h(\D_n^\g/2)= (2x_n)^{1/\g} h(x_n), \quad x_n:=\D_n^\g/2. 
\]
Hence, if $1/\g > \b$, equivalently $\g\in (0,\b^{-1})$, then $R^{(2)}_n\to 0$ by the definition of $\b\ge 0$. Then the result follows. 
\end{proof}

\begin{cor}
Suppose that $\int_{|z|\le 1}|z|\nu(\df z) < \infty$, and that the L\'evy characteristic $\psi$ is given by 
\[
\psi(u) =  \frac{\s^2}{2}u^\top u + \int_{\R^d}(e^{iu^\top z} - 1)\,\nu(\df z). 
\]
Note that $b- \int_{|z|> x} \frac{z}{1 + |z|^2}\,\nu(\df z) = 0$ in the expression \eqref{Q:levy}. 
Then $\b$ in Proposition \ref{prop:1}, \eqref{beta} is consistent with the Blumenthal-Getoor index: 
\[
\b=\inf\l\{\eta>0:\int_{|z|\le 1} |z|^\eta\,\nu(\df z) < \infty\r\} \le 1. 
\]
\end{cor}
\begin{proof}
Note that, under the assumption, we have 
\[
h(x) = \int_{|z|>x}\nu(\df z) + x^{-2}\int_{|z|\le x} |z|^2\,\nu(\df z) + x^{-1}\l|\int_{|z|\le x} \frac{z|z|^2}{1 + |z|^2}\,\nu(\df z)\r|. 
\]
Let $\b_0$ be the Blumenthal-Getoor index: 
\[
\b_0=\inf\l\{\eta\in [0,2]:\int_{|z|\le 1} |z|^\eta\,\nu(\df z) < \infty\r\}. 
\]
Then we easily see by the direct computation that, for any $\e>0$
\[
\lim_{x\to 0}x^{\b_0+\e} h(x) =0,\quad \liminf_{x\to 0}x^{\b_0-\e} h(x) =\infty. 
\]
Indeed, for $x\in (0,1)$, we have 
\begin{align*}
|x^{\b_0 + \e}h(x)|&\le |x|^{\b_0 + \e}\l|\int_{x < |z| \le 1}\nu(\df z) + \int_{|z|>1}\nu(\df z)\r| \\
&\qquad + |x|^{\b_0 + \e-2}  \int_{|z|\le x} |z|^2\,\nu(\df z) + |x|^{\b_0 + \e-1}  \int_{|z|\le x} \frac{|z|^3}{1 + |z|^2}\,\nu(\df z) \\
&\le |x|^{\e}\int_{|z| \le 1}|z|^{\b_0}\nu(\df z) + \l|x^{\b_0 + \e}\int_{|z|>1}\nu(\df z)\r| + |x|^{\e}\int_{|z|\le x} |z|^{\b_0}\,\nu(\df z) \\
&= O(|x|^\e)\to 0. 
\end{align*}
The other condition: $\dis \liminf_{x\to 0}x^{\b_0-\e} h(x) =\infty$ is easier since 
\[
x^{\b_0-\e}h(x) > x^{\b_0 - \e -2}\int_{|z|\le x} |z|^2\,\nu(\df z) \ge x^{-\e}\int_{|z|\le x} |z|^{\b_0}\,\nu(\df z) \to \infty, 
\]
as $x\to 0$. Hence we get $\b_0=\b$. 

\end{proof}

The condition Q2 can be reduced to moment conditions for the L\'evy measure.  
\begin{prop}\label{prop:Q2-q}
Suppose that $Q$ is a L\'evy process with characteristic \eqref{Q:levy}. Then the condition Q2[$q$] holds true if
\[
\int_{|z|>1}|z|^{q}\,\nu(\df z) < \infty. 
\]
\end{prop}
\begin{proof}
The L\'evy process $Q$ has the following L\'evy-Ito decomposition: 
\[
Q_t = a t + b W_t + \int_0^t \int_{|z|\le 1} z\,\wt{N}(\df s,\df z)  + \int_0^t \int_{|z|>1}z\,N(\df s,\df z)
\]
where $a$ and $b$ are some constants, $W$ is a Wiener process, $N$ is a Poisson random measure, and $\wt{N}(\df s,\df z) = N(\df s,\df z) - \nu(\df z) \df s$. 
Then the decomposition \eqref{doob} is given by 
\[
A_t =  a t + \int_0^t \int_{|z|>1}z\,N(\df s,\df z),\quad M_t=  b W_t + \int_0^t \int_{|z|\le 1} z\,\wt{N}(\df s,\df z). 
\]
and it follows from the Burkholder-Davis-Gundy inequality that 
\begin{align*}
\E[TV(A)^q] &\lesssim |a|^q + \E\l|\int_0^1 \int_{|z|>1}|z|\,N(\df s,\df z)\r|^q \\
&\lesssim 1 + \E\l|\int_0^1 \int_{|z|>1}|z|\,\wt{N}(\df s,\df z)\r|^q + \int_{|z|>1}|z|^q\,\nu(\df z). 
\end{align*}
According to the argument as in Bichteler and Jacod \cite{bj96}; see also the proofs of lemma 4.1 and Proposition 3.1 by Shimizu and Yoshida \cite{sy06}, 
we see that 
\[
\E\l|\int_0^1 \int_{|z|>1}|z|\,\wt{N}(\df s,\df z)\r|^q \lesssim \int_{|z|>1}|z|^q\,\nu(\df z). 
\]
Hence 
\[
\E[TV(A)^q] \lesssim \int_{|z|>1}|z|^q\,\nu(\df z) < \infty. 
\]
By the similar argument, it is easy to see by H\"older's inequality that, for any $m>1$, 
\[
\E|[M,M]|^{q/2} \le \l(\E|[M,M]|^{mq/2}\r)^{1/m}\lesssim 1 +  \l(\int_{|z|\le 1}|z|^{mq/2}\,\nu(\df z)\r)^{1/m} < \infty,  
\]
if we take $mq/2\ge 2$. This completes the proof. 
\end{proof}

\section{Numerical study} \label{sec:simulation}

\subsection{2-dim model}\label{sec:2-dim}
We consider the following 2-dimensional L\'evy driven SDE: 
\begin{align}
b(x,\th) = \l(\sqrt{\th_1 + x_1^2 + x_2^2}, -\frac{\th_2x_2}{\sqrt{1 + x_1^2 + x_2^2}}\r)^\top,\quad Q_t=\begin{pmatrix}V_t^{\k,\xi}+B_t\\ S_t^\alpha\end{pmatrix}, \label{sim:ex1}
\end{align}
where $B$ is a standard Brownian motion, $S^\a$ is a standard symmetric $\alpha$-stable process $S_\alpha(1,0,0)$, and
$V^{\k,\xi}$ is a variance gamma process with L\'evy density
\begin{align*}
p_V(z)=\frac{\k}{|z|}e^{-\xi|z|},\quad z\in \R,\ \k,\xi>0, 
\end{align*}
which is obtained by Brownian subordination with a gamma process $G_t\sim \Gamma(\mathrm{shape}=ct,\mathrm{scale}=1/\lambda)\ (c,\lambda>0)$ as follows:
\[V_t^{\k,\xi}=\sigma W_{G_t}\ (\sigma>0),\quad \k=\frac{\lambda^2}{c},\ \xi=\frac{\sqrt{2\k}}{\sigma} \]
where $W$ is the standard Brownian motion independent of $G$; see, e.g., Cont and Tankov \cite{ct04} for details.
Assume that $W,S^\a$ and $V^{\k,\xi}$ are independent of each other.

In the sequel, we set values of parameters as
\[
(X_0^{(1)},X_0^{(2)})=(1,1),\quad (\th_1,\th_2)=(2,1),\quad (\k,\xi,\alpha)=(5,3,3/2).
\]
A sample path is given in Figure \ref{fig:path}. 
Then both $X^{(1)}$ and $X^{(2)}$ are unbounded variation jump-processes with finite activity of jumps for $X^{(1)}$ and infinite for $X^{(2)}$. 
We will compare our threshold-type estimator to the LSE by Long {\it et al.} \cite{lss13}. 

\subsection{LSE and threshold estimator}

Note that the LSE $\wh{\th}^{LSE}=(\wh{\th}_{1,n,\e}^{LSE},\wh{\th}_{2,n,\e}^{LSE})$ is a solution to 
\begin{align*}
\sum_{k=1}^n \frac{\DX^{(1)}}{\sqrt{\wh{\th}_{1,n,\e}^{LSE} + (X_{\t}^{(1)})^2 + (X_{\t}^{(2)})^2}} = 1;\quad
\wh{\th}_{2,n,\e}^{LSE}=-\frac{\sum_{k=1}^{n}\frac{\l(\DX^{(2)}\r) X_{\t}^{(2)}}{\sqrt{1+(X_{\t}^{(1)})^2+(X_{\t}^{(2)})^2}}}{n^{-1}\sum_{k=1}^{n}\frac{(X_{\t}^{(2)})^{2}}{1+(X_{\t}^{(1)})^2+(X_{\t}^{(2)})^2}}.  \label{eq:theta1}
\end{align*}
and that our estimator $\wh{\th}=(\wh{\th}_{1,n,\e},\wh{\th}_{2,n,\e})$ is a solution to 
\begin{align*}
\sum_{k=1}^n \frac{\DX^{(1)}}{\sqrt{\wh{\th}_{1,n,\varepsilon} + (X_{\t}^{(1)})^2 + (X_{\t}^{(2)})^2}} \I_{\{|\DX|\le \d_{n,\e}\}}= \frac{1}{n}\sum_{k=1}^n \I_{\{|\DX|\le \d_{n,\e}\}}
\end{align*}
\begin{align*}
\wh{\th}_{2,n,\e}=-\frac{\sum_{k=1}^{n}\frac{\l(\DX^{(2)}\r)X_{\t}^{(2)}}{\sqrt{1+(X_{\t}^{(1)})^2+(X_{\t}^{(2)})^2}}\I_{\{|\DX|\le \d_{n,\e}\}}}{n^{-1}\sum_{k=1}^{n}\frac{(X_{\t}^{(2)})^{2}}{1+(X_{\t}^{(1)})^2+(X_{\t}^{(2)})^2}\I_{\{|\DX|\le \d_{n,\e}\}}}.  
\end{align*}

In the simulations, we tried $n=1000, 3000, 5000$ and $\e=(0.4,0.3,0.05)$. 
The results for the LSE given in \eqref{lse1} are in Table \ref{tab:lse}, 
and those for the threshold-type estimator in \eqref{tte1} are in \ref{tab:tte1} and \ref{tab:tte2}, 
in which the threshold $\d_{n,\e}=\e/5$ and $\e/10$ are used, respectively. 
The simulations are iterated 10000 times, and mean and standard deviation (s.d.) of estimators are given in those tables. 
In Table \ref{tab:tte1} and \ref{tab:tte2}, the values of $n\d_{n,\e}$ are also included. 

\subsection{Discussion}
From numerical results, we can observe that the threshold-type estimator improves the accuracy of estimation in the sense of 
the standard deviation compared with the usual LSE by Long {\it et al.} {\cite{lss13}. Especially, the improvement for $\th_2$ is more drastic than that for $\th_1$. 
This would be because $\th_2$ is the mean-reverting parameter for $X^{(2)}$, a {\it stable process} which has much more frequent and larger jumps than those of $X^{(1)}$. 
In order to make the estimation of $\th_1$ more accurate, we need to make $\e$ smaller rather than $n$ larger. 

We should note that the case where $(n,\e,\d_{n,\e})=(1000,0.05,\e/10)$ causes a large bias, the reason of which would be that 
the asymptotic theory does not work well because $n\d_{n,\e}=5$ is `small' although it should be large enough in the theory. 
From many simulations omitted here, we see that, at least, ``$n\d_{n,\e}\ge 10$"  would be needed for estimation with `small' bias; see also, e.g., the case where $(n,\e,\d_{n,\e})=(1000,0.05,\e/5)$, which returns better estimation. 

Although there remains a problem to choose $\d_{n,\e}$ in practice, we can use the method proposed by Shimizu \cite{s10} if the parameters in noise process is known or estimable.

Finally, we observe normal QQ-plots for normalized estimators $\e^{-1}(\wh{\th}_{1,n,\e} - \th_1)$ and $\e^{-1}(\wh{\th}_{2,n,\e} - \th_2)$ in the case where 
$n=5000$, $\e=0.1$ and $\d=\e/5$; see Figures \ref{fig:qq1} and \ref{fig:qq2}. 
According to the results, the estimators with $\d=\e/5$ seem asymptotically normal although Theorem \ref{thm:asym-dist} does not necessarily say that. 
For your reference, see Figures \ref{fig:qq0-1} and \ref{fig:qq0-2} that are the normal QQ-plots for the LSE without filter proposed by Long {\it et al.} \cite{lss13}. 
The figures show that the usual LSE's are not necessarily asymptotically normal as the theory saying. Therefore, the asymptotic normality-like phenomena would be due to the filter effects. 

We could understand those results intuitively as follows: cutting large jumps from a process with infinite activity jumps, the remaining small jumps will behave as a Brownian motion. 
For example, suppose that the driving noise $Q$ is a L\'evy process of infinite activity jumps with the L\'evy measure $\nu$, and put 
\[
Q^{(\d)}_t:=\int_0^t\int_{|z|\le \d} z\,\wt{N}(\df s,\df z), \quad \d>0, 
\]
where $\wt{N}$ is a compensated Poisson random measure as given in the proof of Proposition \ref{prop:Q2-q}. 
Then, according to Asmussen and Rosinski \cite{ar01}, it follows for $\s^2(\d):=\int_{|z|\le \d} |z|^2\,\nu(\df z)$ that  
\begin{align*}
\s(\d)^{-1}Q^{(\d)} \toD B, 
\end{align*}
in $\mathbb{D}[0,1]$-space equipped with the sup-norm under a certain assumption, where $B$ is a standard Brownian motion. 
Therefore, by an appropriate norming of estimator, the limit in Theorem \ref{thm:asym-dist} might be an integral with respect to the process that is close to a Brownian motion. 
This consideration indicates that the LSE with filter can be ``approximately" asymptotically normal by letting $\d_{n,\e}$ converge to zero in a suitable rate, which is a great advantage 
when we would like to make confidence intervals or do statistical testing. 

In the next subsection, we shall try to understand this phenomenon theoretically.

\section{Could a filtered LSE be asymptotically normal?} \label{sec:normal?}

In this section, we assume that $Q^\e\equiv Q$ is a L\'evy process with L\'evy measure $\nu$. 
We consider the following two cases for $Q$ with characteristic \eqref{Q:levy}: 
\begin{itemize}
\item Finite activity case: $\dis \int_{|z|\le 1} \nu(\df z) < \infty$ and $\s^2>0$; 
\item Infinite activity case: $\dis \int_{|z|\le 1} \nu(\df z) = \infty$\ (possibly, $\s^2=0$).  
\end{itemize}
The former case, it would be possible to show the asymptotic normality of the filtered LSE by separating the increments $\DX$'s with or without jumps as in Shimizu and Yoshida \cite{sy06}. 
However, when  $\int_{|z|\le 1} \nu(\df z) = \infty$, it is known that the filter $\filter$ is not enough to separate $\DX$'s with or without jumps; 
see Shimizu \cite{s06}, Lemma 3.3 and some remarks on that. 
In this section, we shall consider the following {\it ad hoc} situation to understand why the filtered LSE looks like asymptotically normal.  
\begin{flushleft}
{\bf [Assumption]} {\it All the jumps of $Q$ are observed.} 
\end{flushleft}
Under this assumption, the following contrast function does make sense. 
\begin{align*}
\wt{\Phi}_{n,\e,\d}(\th)=\e^{-2}\D_n^{-1}\sum_{k=1}^n |\DX - b(X_\t,\th)\cdot \D_n|^2\I_{\l\{\|\D Q^\e\|^*_k \le \d\r\}},  
\end{align*}
where $\d >0$ is a constant, and 
\[
\|Q^\e\|^*_k = \sup_{t\in (t_k,t_{k+1}]}|\D Q_t^\e|,\quad \D Q_t^\e = \e \cdot (Q_t- Q_{t-}). 
\]
Under our assumption, we can specify if $\I_{\l\{\|\D Q^\e\|^*_k \le \d\r\}}=1$ or $0$, and we can define the estimator of $\th$ as the minimum contrast estimator: 
\[
\wt{\th}_{n,\e,\d} = \arg\min_{\th\in \Theta} \wt{\Phi}_{n,\e,\d}(\th). 
\]
Hereafter, we further use the following notation: 
\begin{itemize}
\item For a $\k>0$, 
\[
\s^2(\k):=\int_{|z|\le \k} |z|^2\,\nu(\df z);\quad \la(\k):=\int_{|z|>\k}\nu(\df z),\]
where $\nu$ is a L\'evy measure of $Q$. 

\item All the asymptotic symbols are used under $\d,\e\to 0$ and $n\to \infty$. 
\end{itemize}

\subsection{Finite activity case}

Suppose that $\dis \int_{|z|\le 1} \nu(\df z) < \infty$ and $\s^2>0$, which implies that $Q$ is written as 
\begin{align}
Q_t = \s W_t + \sum_{i=1}^{N_t} Y_i, \label{FA-Q}
\end{align}
where $N$ is a Poisson process with intensity $\la:=\int_{\R} \nu(\df z)$ and $\{Y_i\}_{i=1,2,\dots}$ is an i.i.d. sequence with distribution $\la^{-1}\nu$. 
In this special case, we have the following result by taking $\d_{n,\e}\downarrow 0$ faster than the speed of $\e$ that is a ``magnitude of jumps". 

\begin{thm}\label{thm:finite-activity}
Suppose that $Q$ is given by \eqref{FA-Q}, and that A1--A4 hold true. 
Moreover, suppose that 
\begin{align}
\d/\e\to 0.\label{d'-1}
\end{align}
Then 
\[
\e^{-1}(\wt{\th}_{n,\e,\d} - \th_0) \toP I^{-1}(\th_0)\int_0^1 \n_\th b(X_t^0,\th_0)[\df W_t],  
\]
Hence $\wt{\th}_{n,\e,\d}$ is asymptotically normal. 
\end{thm}

\subsection{Infinite activity case} 

\begin{thm}\label{thm:infinite-activity} 
Let $Q^\e \equiv Q$ be a L\'evy process with $\int_{|z|<1}\nu(\df z) = \infty$, and suppose A1--A4, Q1[$\g$], and that Q2[$q$] holds true for any $q>0$. 
Moreover suppose that  
\begin{align}
\frac{\la(\d/\e)}{n\log n} \to c \in (0,1),\quad  n\e\D_n^\g\to 0, \label{d'-2}
\end{align}
and that there exists a constant $\rho\in (0,1)$ such that 
\begin{align}
\s^{\rho}(\d/\e) \log n &\to \infty; \label{with-rho} \\
n\e \cdot \s(\d/\e)&\to \infty.  \label{ne} 
\end{align}
Furthermore, suppose for each $\k>0$ that  
\begin{align}
\s\l(\k \s(\d/\e) \wedge \d/\e\r) &\sim \s(\d/\e), \label{sigma(delta)-1}
\end{align}
Then there exists a $d$-dimensional Brownian motion $B$, independent of $X_0=x$, such that 
the following weak convergence holds true:  
\[
\l(\s(\d/\e)\e\r)^{-1}\l(\wt{\th}_{n,\e,\d} - \th_0\r)\toD I^{-1}(\th_0)\int_0^1 \n_\th b(X_t^0,\th_0)\l[\df B_t\r],
\]
Hence $\wt{\th}_{n,\e,\d}$ is asymptotically normal. 
\end{thm}

We shall give a concrete example that satisfies the above situation.

\begin{remark}
The assumption \eqref{sigma(delta)-1} is to approximate a component of compensated small jumps by a Wiener process; see Theorem 2 by Asmussen and Rosinski \cite{ar01}. 
According to Proposition 2.1 in \cite{ar01}, a simple condition that 
\begin{align}
(\d/\e)^{-1}\s(\d/\e)\to \infty \label{sigma(delta)-3}
\end{align}
is sufficient for \eqref{sigma(delta)-1}. Moreover, note that \eqref{sigma(delta)-1} requires a high jump-activity, 
which excludes cases where $Q$ is a compound Poison process or a gamma process; see some Examples 2.2--2.4 in \cite{ar01}. 
Therefore it is assumed in \eqref{with-rho} that 
\[
\la(\d/\e)\to \infty\quad (\d\to 0). 
\]

\end{remark}

\begin{remark}
Although this is the result under an ideal situation that all the jumps are observable, 
we can imagine that such a phenomenon `approximately' occurs in simulations presented in Figures \ref{fig:qq1} and \ref{fig:qq2}. 
That is, when $\D_n$ and $\d_{n,\e}$ are sufficiently `small', a filter $\I_{\{|\DX|\le \d_{n,\e}\}}$ can successfully cut jumps whose sizes are larger than 
$\d_{n,\e}$. 
However, it would be hard to show the similar result when the observations are completely discrete as in our original setting because the filter $\I_{\{|\DX|\le \d_{n,\e}\}}$ 
can not exactly exclude an increment that includes the jumps whose sizes are larger than $\d_{n,\e}$; see also Shimizu \cite{s06} about the filter in infinite activity cases. 
The complete analysis in the discretely observed cases would be an important work in the future. 
\end{remark}

\begin{table}[htbp] 
\begin{center}
\begin{tabular}{rrrrr} \hline
\quad $\epsilon=0.4$\quad  &\quad  $n=1000$\quad  &\quad  $n=3000$\quad  &\quad  $n=5000$\quad  & True \\\hline
$\wh{\th}_{1,n,\e}^{LSE}$ & 2.54872 & 2.53364 & 2.78911 & 2.0 \\
(s.d.) & (2.4370) & (2.4489) & (2.5602) &  \\\hline
$\wh{\th}_{2,n,\e}^{LSE}$ & 1.72415 & 1.79755 & 1.78645 & 1.0 \\
(s.d.) & (3.5426) & (4.2814) & (2.9579) &  \\\hline
 &  &  &  &  \\
 &  &  &  &  \\\hline
$\epsilon=0.3$ & $n=1000$ & $n=3000$ & $n=5000$ & True \\\hline
$\wh{\th}_{1,n,\e}^{LSE}$ & 2.31618 & 2.29381 & 2.34757 & 2.0 \\
(s.d.) & (1.8248) & (1.7926) & (1.7429) &  \\\hline
$\wh{\th}_{2,n,\e}^{LSE}$ & 1.50664 & 1.5275 & 1.53632 & 1.0 \\
(s.d.) & (2.8685) & (2.7667) & (2.8160) &  \\\hline
 &  &  &  &  \\
 &  &  &  &  \\\hline
$\epsilon=0.05$ & $n=1000$ & $n=3000$ & $n=5000$ & True \\\hline
$\wh{\th}_{1,n,\e}^{LSE}$ & 2.00599 & 2.01071 & 2.01002 & 2.0 \\
(s.d.) & (0.2951) & (0.2913) & (0.2938) &  \\\hline
$\wh{\th}_{2,n,\e}^{LSE}$ & 1.05963 & 1.04438 & 1.06135 & 1.0 \\
(s.d.) & (1.3026) & (0.6773) & (0.7344) &  \\\hline
\end{tabular}
\end{center}
\caption{These are results for the LSE (without filter) based on Long {\it et al.} \cite{lss13}. 
We find that the standard deviation (s.d.) are large, especially, for $\wh{\th}_{2,n,\e}$, which implies an {\it unstability} of estimation. 
We would like to improve the stability by using a `filter'.}
\label{tab:lse}
\end{table}

\begin{figure}[htbp]
\begin{center}
\includegraphics[width=7cm,height=8cm]{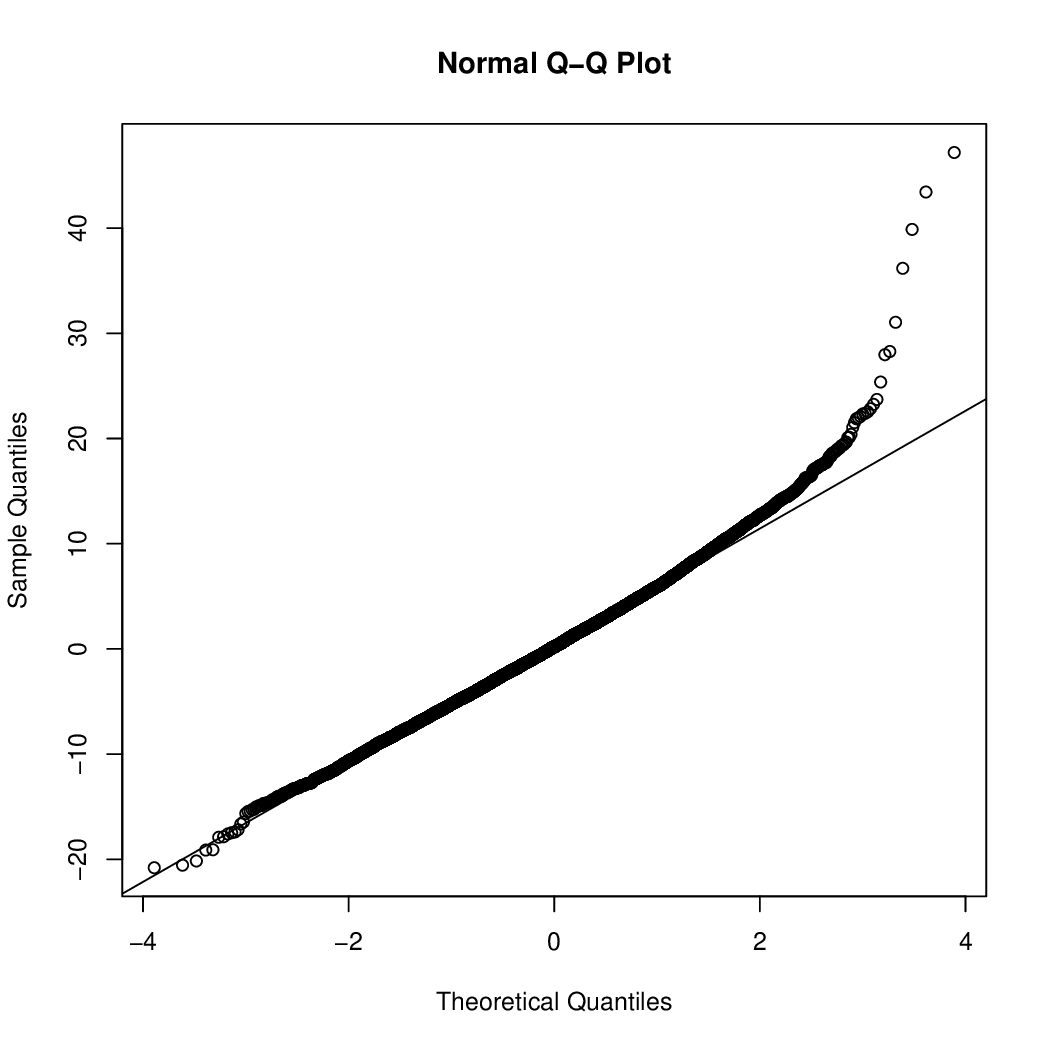}
\includegraphics[width=7cm,height=8cm]{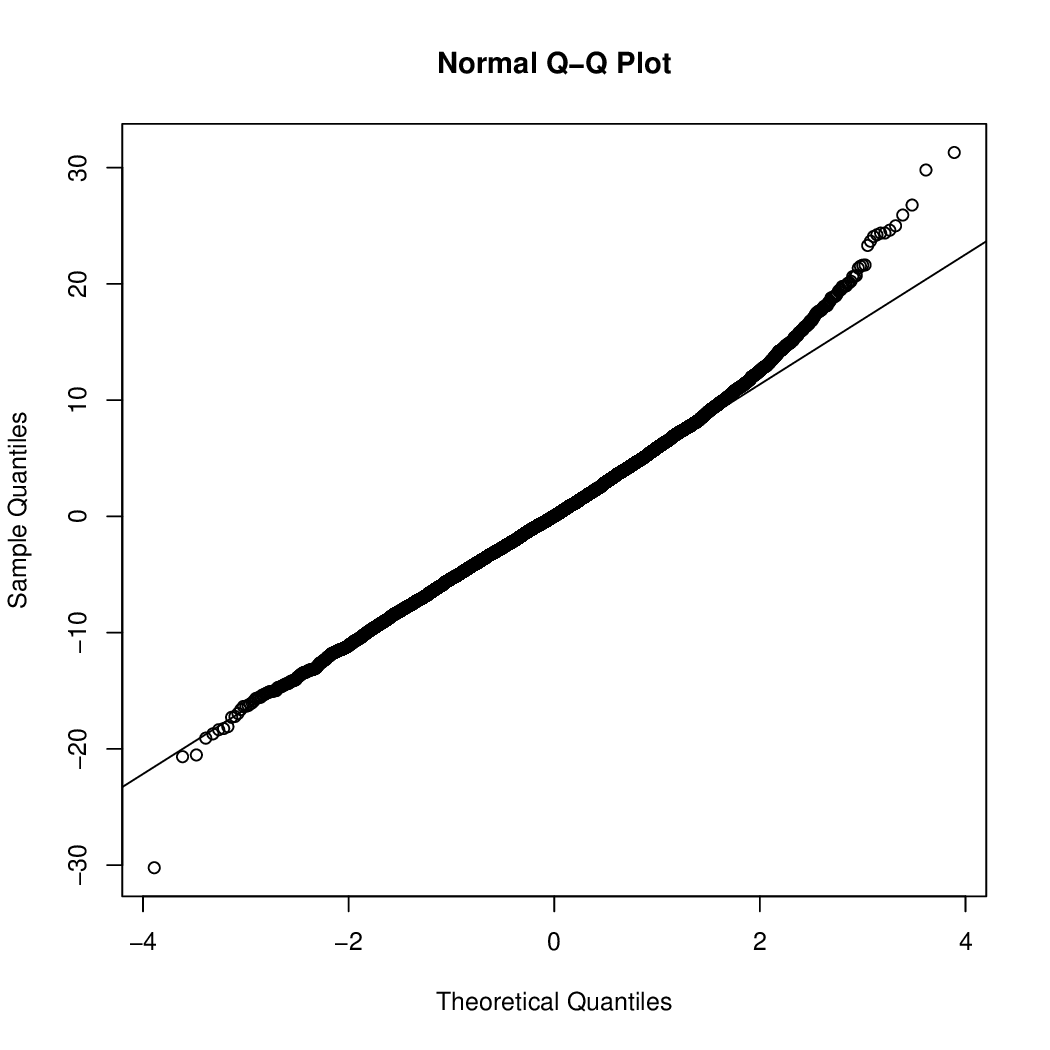}
\caption{Normal QQ-plots for $\wh{\th^{LSE}}_{1,n,\e}$. Left: $\e=0.1$, Right: $\e=0.05$. 
The results show that the right tail is especially heavier than that of normal distribution. 
}
\label{fig:qq0-1}
\end{center}
\end{figure}

\begin{figure}[htbp]
\begin{center}
\includegraphics[width=7cm,height=8cm]{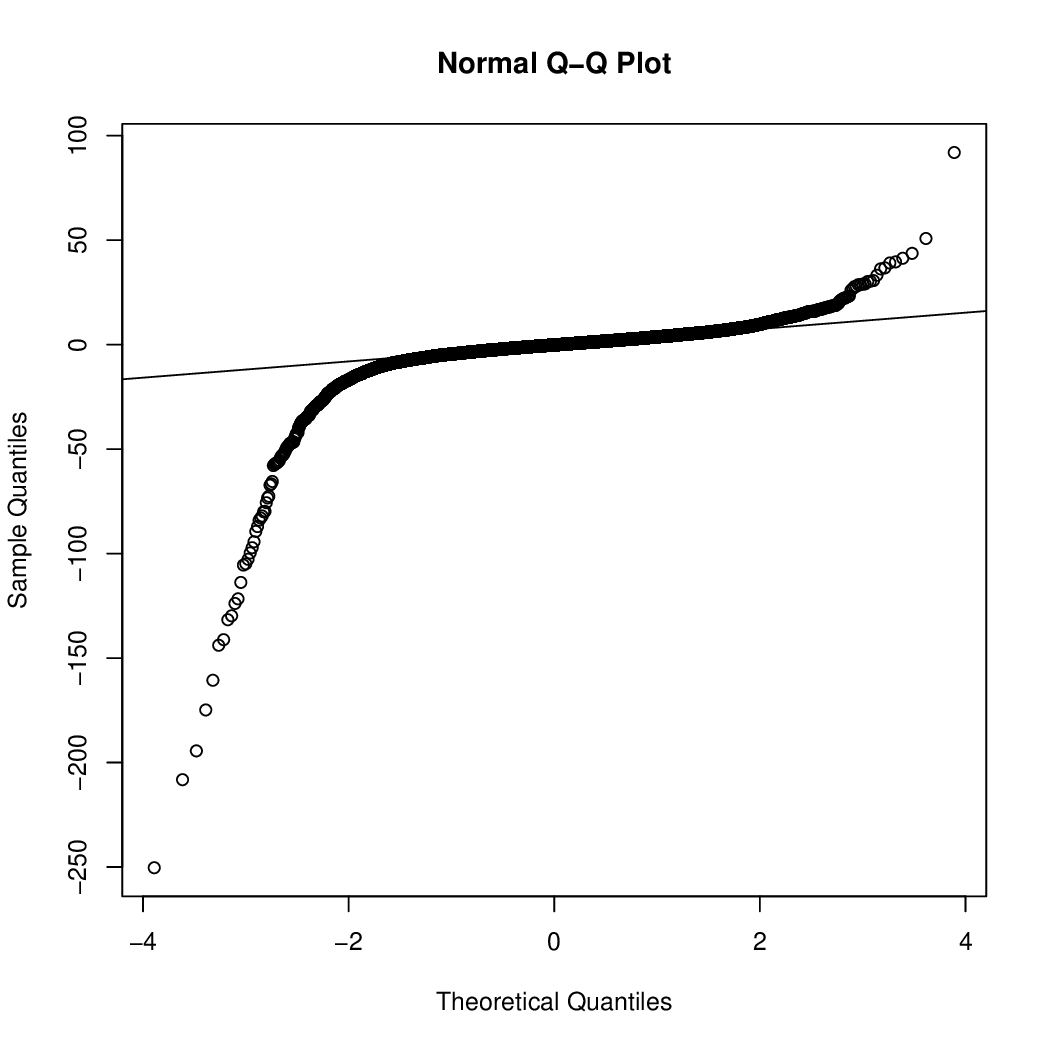}
\includegraphics[width=7cm,height=8cm]{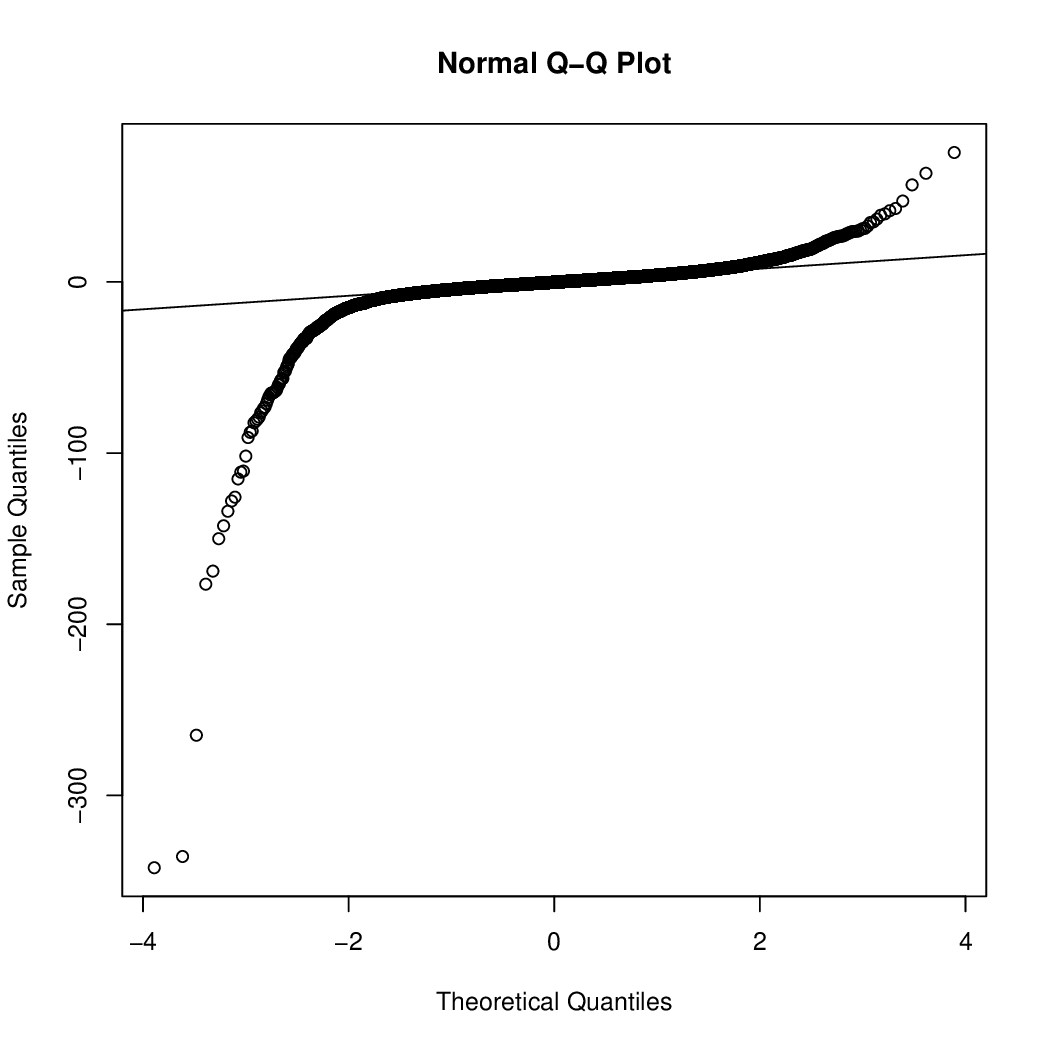}
\caption{Normal QQ-plots for $\wh{\th^{LSE}}_{2,n,\e}$. Left: $\e=0.1$, Right: $\e=0.05$. 
The results show that the both of tails are heavier than those of normal distribution.}
\label{fig:qq0-2}
\end{center}
\end{figure}

\begin{table}[htbp]
\begin{center}
\begin{tabular}{rrrrr} \hline
\quad $\epsilon=0.4$\quad  & \quad $n=1000$\quad & \quad $n=3000$\quad  & \quad $n=5000$\quad  & True \\\hline
$\wh{\th}_{1,n,\e}$ & 2.46071 & 2.45564 & 2.43215 & 2.0 \\
(s.d.) & (2.1968) & (2.1777) & (2.1860) &  \\\hline
$\wh{\th}_{2,n,\e}$ & 1.14289 & 1.15283 & 1.16222 & 1.0 \\
(s.d.) & (0.8721) & (0.88426) & (0.8856) &  \\\hline
 &  &  &  &  \\
 &  &  &  &  \\\hline
$\epsilon=0.3$ & $n=1000$ & $n=3000$ & $n=5000$ & True \\\hline
$\wh{\th}_{1,n,\e}$ & 2.25007 & 2.25972 & 2.25829 & 2.0 \\
(s.d.) & (1.6149) & (1.6121) & (1.6249) &  \\\hline
$\wh{\th}_{2,n,\e}$ & 1.08105 & 1.1047 & 1.09448 & 1.0 \\
(s.d.) & (0.6498) & (0.6489) & (0.6563) &  \\\hline
 &  &  &  &  \\
 &  &  &  &  \\\hline
$\epsilon=0.05$ & $n=1000$ & $n=3000$ & $n=5000$ & True \\\hline
$\wh{\th}_{1,n,\e}$ & 2.00619 & 2.00827 & 2.00681 & 2.0 \\
(s.d.) & (0.2623) & (0.2594) & (0.2652) &  \\\hline
$\wh{\th}_{2,n,\e}$ & 0.98972 & 0.99936 & 1.00039 & 1.0 \\
(s.d.) & (0.10098) & (0.1031) & (0.1037) &  \\\hline
\end{tabular}
\end{center}
\caption{Results with filter: $\d_{n,\e} = \e/5$. Compared with the LSE, the improvement for s.d. of $\wh{\th}_{2,n,\e}$ is drastic although the one for $\wh{\th}_{1,n,\e}$ is less. 
To make the estimation of $\th_1$ accurate, we need to make $\e$ smaller. 
The values of $(n\d_{n,\e}, \d_{n,\e}\e^{-1}n^{1/4})$ seems enough to meet the asymptotic conditions such that they must tend to $(\infty,0)$. }
\label{tab:tte1}
\end{table}

\begin{table}[htbp]
\begin{center}
\begin{tabular}{rrrrr} \hline
\quad $\epsilon=0.4$\quad  & \quad $n=1000$\quad  & \quad $n=3000$\quad  & \quad $n=5000$\quad  & True \\\hline
$\wh{\th}_{1,n,\e}$ & 2.10973 & 2.24022 & 2.46723 & 2.0 \\
(s.d.) & (2.0587) & (2.1736) & (2.2064) &  \\\hline
$\wh{\th}_{2,n,\e}$ & 1.06597 & 1.10031 & 1.10585 & 1.0 \\
(s.d.) & (0.6997) & (0.7226) & (0.7352) &  \\\hline
 &  &  &  &  \\
 &  &  &  &  \\\hline
$\epsilon=0.3$ & $n=1000$ & $n=3000$ & $n=5000$ & True \\\hline
$\wh{\th}_{1,n,\e}$ & 1.92432 & 2.25282 & 2.24733 & 2.0 \\
(s.d.) & (1.4600) & (1.6221) & (1.6021) &  \\\hline
$\wh{\th}_{2,n,\e}$ & 1.02419 & 1.05871 & 1.05869 & 1.0 \\
(s.d.) & (0.5216) & (0.5453) & (0.5349) &  \\\hline
 &  &  &  &  \\
 &  &  &  &  \\\hline
$\epsilon=0.05$ & $n=1000$ & $n=3000$ & $n=5000$ & True \\\hline
$\wh{\th}_{1,n,\e}$ & 0.78959 & 2.00133 & 2.00864 & 2.0 \\
(s.d.) & (0.2165) & (0.2628) & (0.2653) &  \\\hline
$\wh{\th}_{2,n,\e}$ & 0.84069 & 0.98834 & 0.99450 & 1.0 \\
(s.d.) & (0.0754) & (0.0853) & (0.0865) &  \\\hline
\end{tabular}
\end{center}
\caption{Results with filter: $\d_{n,\e} = \e/10$. The s.d. for $\wh{\th}_{2,n,\e}$ is smaller than those with $\d_{n,\e}/5$ as well as those for $\wh{\th}_{1,n,\e}$. 
However, we should be careful by observing the case of $(n,\e)=(1000,0.05)$, where both estimators are negatively biased. This would be because $n\d_{n,\e}=5$ is too small 
to meet the corresponding asymptotic condition: $n\d_{n,\e}\to \infty$.}
\label{tab:tte2}
\end{table}

\begin{figure}[htbp]
\begin{center}
\includegraphics[width=7cm,height=8cm]{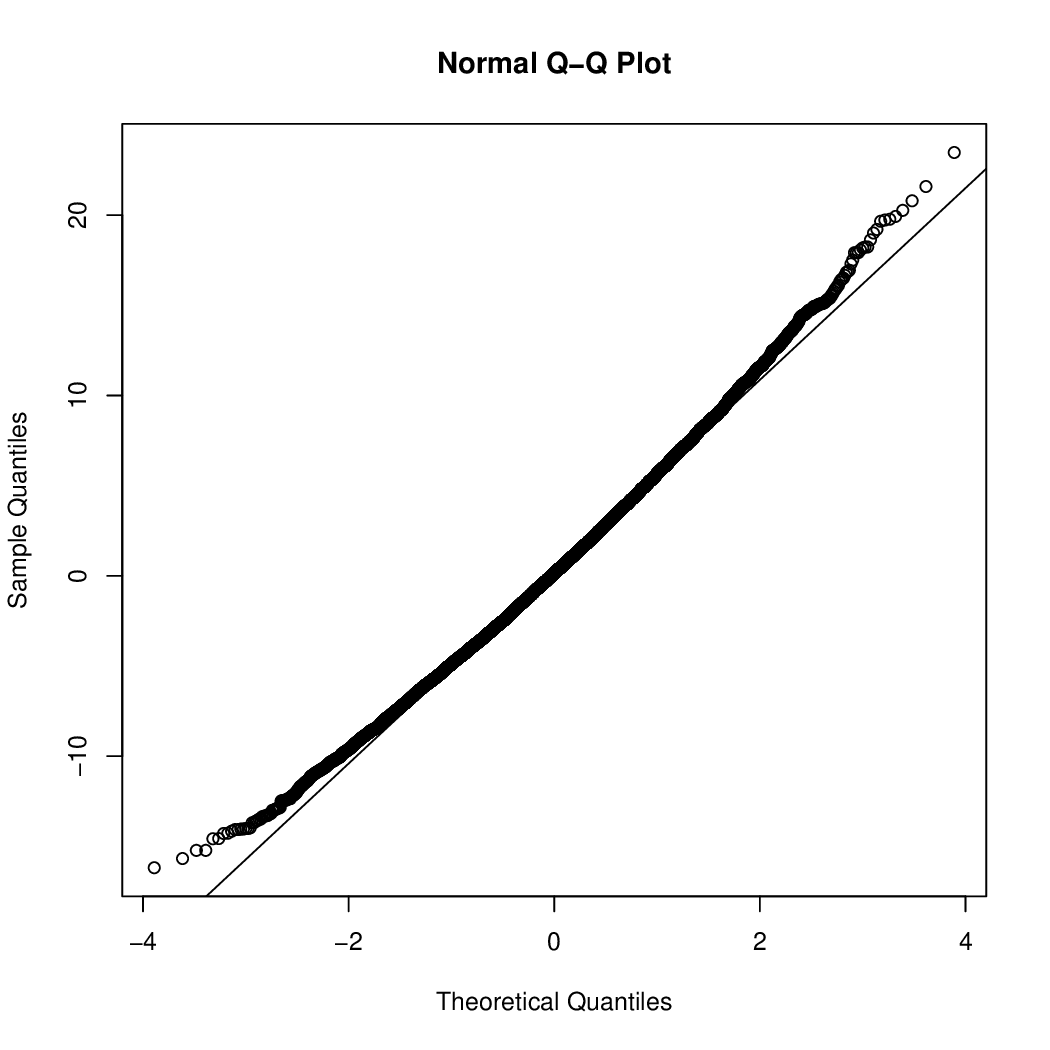}
\includegraphics[width=7cm,height=8cm]{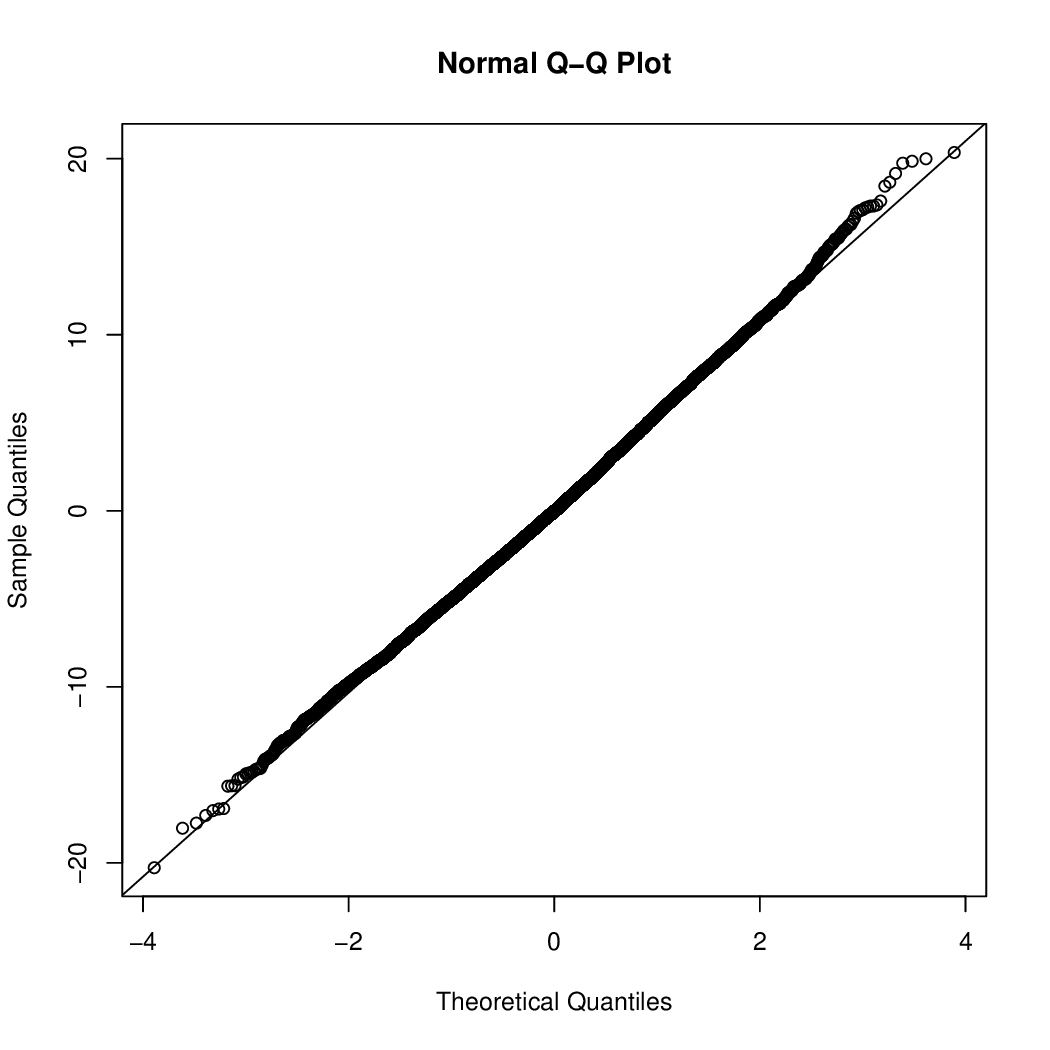}
\caption{Normal QQ-plots for $\wh{\th}_{1,n,\e}$. Left: $\e=0.1$, Right: $\e=0.05$. 
When $\e$ is small such as $\e=0.05$, the distribution of $\wh{\th}_{1,n,\e}$ seems almost Gaussian.
}
\label{fig:qq1}
\end{center}
\end{figure}

\begin{figure}[htbp]
\begin{center}
\includegraphics[width=7cm,height=8cm]{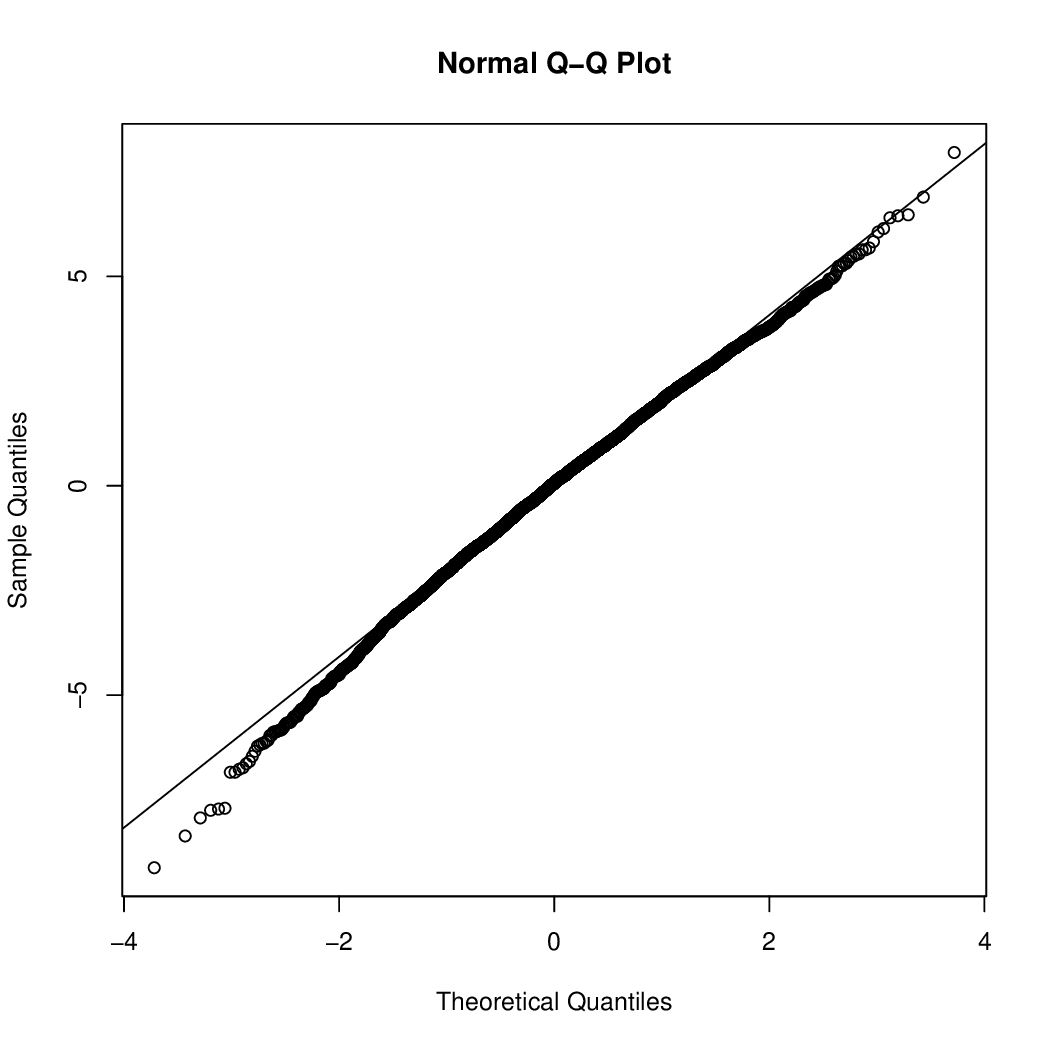}
\includegraphics[width=7cm,height=8cm]{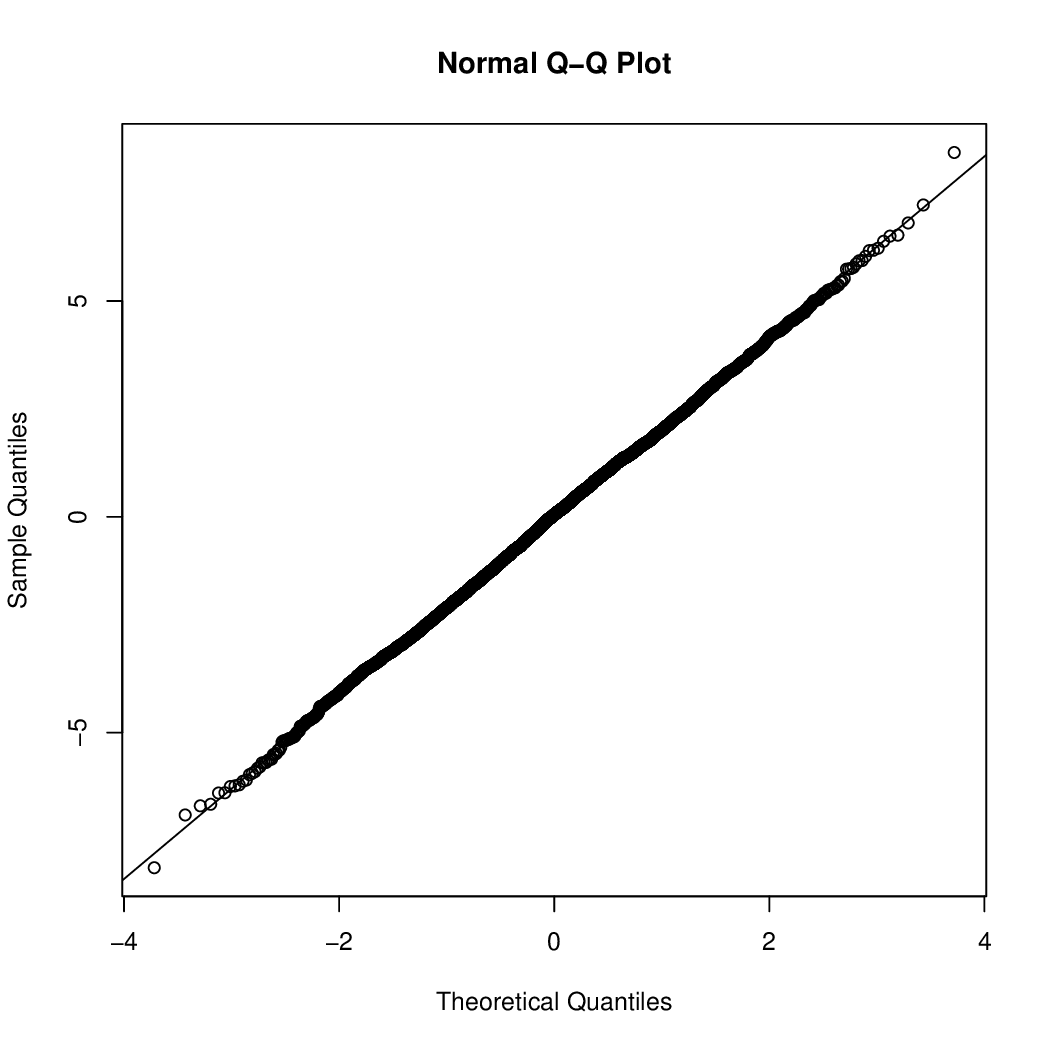}
\caption{Normal QQ-plots for $\wh{\th}_{2,n,\e}$. Left: $\e=0.1$, Right: $\e=0.05$. 
When $\e$ is small such as $\e=0.05$, the distribution of $\wh{\th}_{1,n,\e}$ seems almost Gaussian.}
\label{fig:qq2}
\end{center}
\end{figure}

\begin{figure}[htbp]
\begin{center}
\includegraphics[width=10cm,height=8cm]{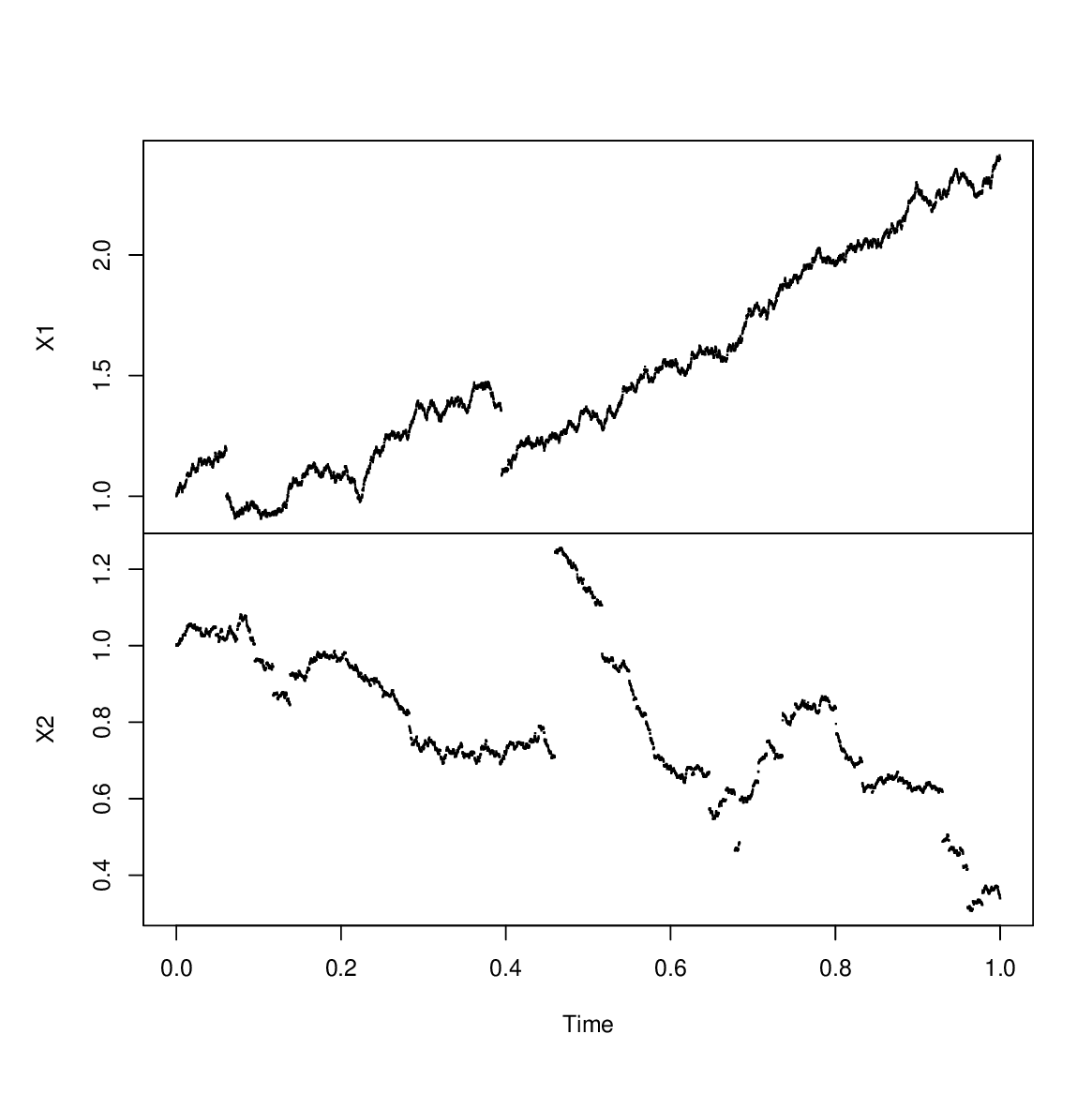}
\caption{A sample path of Model \eqref{sim:ex1} with $(\theta_1,\theta_2,\k,\xi,\alpha)=(2,1,5,3,3/2)$ and $\varepsilon=0.4$.}
\label{fig:path}
\end{center}
\end{figure}

\section{Proofs}\label{sec:proof}

\subsection{Preliminary lemmas}
We shall first establish some preliminary lemmas to show the main theorems later. 

To evaluate functions of discrete samples, we use the following notation as in Long {\it et al.} \cite{lss13}: 
\begin{itemize}
\item Discretized process: $Y_t^{n,\e}:= X_{\lfloor nt\rfloor/n}$ for $t\ge 0$ and $n\in \N$, 
where $\lfloor x \rfloor$ stands for the integer part of $x\in \R$. 

\item Since $Q$ is a semimartingale, we can consider the Doob-Meyer decomposition: $Q= A + M$, where $A$ is a process with finite variation, and $M$ is an $\F_t$-local martingale with $A_0=M_0=0\ a.s.$ 

\item A stopping time for localization: for $m\in \N$, 
\[
\tau_m^{n,\e}=\inf\{t\ge 0\,:\,|X_t^0|\wedge |Y_t^{n,\e}| \ge m\} \wedge T_m, 
\]
where $T_m:=\inf\{t\ge 0\,:\,[M,M]_t\wedge \int_0^t |\df A_s| \ge m\}$. As a convention, $\inf \emptyset =\infty$. Hence, note that $\lim_{m\to \infty}T_m=\infty$ almost surely. 
\end{itemize}

\begin{lemma}\label{lem:3.1}
Under A1, it holds that   
\begin{align}
\l\|Y^{n,\e} - X^0\r\|_* \toP 0. \label{lem:3.1-1}
\end{align}
In addition, suppose Q2[q] for some $q>0$. Then it holds that 
\begin{align} 
\E\l[\l\|Y^{n,\e} - X^0\r\|_*^q\r] =O\l(\frac{1}{n^q} + \e^q\r). \label{lem:3.1-2}
\end{align}
\end{lemma}
\begin{proof}
From \eqref{Q}, it follows for $\e$ small enough that 
\begin{align}
\|Q^\e\|_* \lesssim \|Q\|_* + 1.  \label{sup-Q}
\end{align}
Hence we can take the same argument as in the proof of Lemma 3.1, (3.1) by Long {\it et al.} \cite{lss13} to obtain that 
\begin{align}
\l\|X - X^0\r\|_*\lesssim \e \|Q^\e\|_* \lesssim \e (\|Q\|_*+1) \toP 0,  \label{X-X^0}
\end{align}
since $\|Q\|_*$ is bounded in probability. Hence the fact that $\lfloor nt\rfloor/n\to t$ as $n\to \infty$ yields that 
\begin{align*}
\l\|Y^{n,\e} - X^0\r\|_* \le \l\|Y^{n,\e} - X\r\|_* + \l\|X - X^0\r\|_*  \toP 0. 
\end{align*}
This is the proof of \eqref{lem:3.1-1}. 

Moreover, note that 
\begin{align*}
\E\l[\l\|Y^{n,\e} - X^0\r\|_*^q\r] &\lesssim \E\l[\sup_{t\in [0,1]}\l|X_{\lfloor nt\rfloor/n} - X_t\r|^q\r] + \E\l[\l\|X - X^0\r\|_*^q\r] \\
&\lesssim \E\l[\sup_{t\in [0,1]}\l|\int_t^{\lfloor nt\rfloor/n} b(X_s,\th_0)\,\df s\r|^q \r] \\
&\quad + \e^q \E\l[\sup_{t\in [0,1]}\l|Q^\e_{\lfloor nt\rfloor/n} - Q^\e_t\r|^q\r] + \E\l[\l\|X - X^0\r\|_*^q\r]. 
\end{align*}
Noticing the linear growthness of the function $b$: $|b(x,\th)|\lesssim 1 + |x|$, and the inequalities \eqref{sup-Q} and \eqref{X-X^0}, 
we have that
\begin{align*}
\E\l[\l\|Y^{n,\e} - X^0\r\|_*^q\r]
&\lesssim \E\l[\sup_{t\in [0,1]}\l|\int_t^{\lfloor nt\rfloor/n} (1 + \|X\|_*) \,\df s\r|^q \r] \\
&\quad + \e^q \E\l[\|Q\|_*^q\r] + \E\l[\l\|X - X^0\r\|_*^q\r] \\
&\lesssim \sup_{t\in [0,1]}\l[\frac{nt - \lfloor nt\rfloor}{n}\r]^q \E\l[\l(1 + \|X\|_*\r)^q\r] + \e^q\l(1 + \E\l[\|Q\|_*^q\r]\r) \\
&\lesssim \sup_{t\in [0,1]}\l[\frac{nt - \lfloor nt\rfloor}{n}\r]^q + \e^q \l(1 + \E\|A\|_*^q + \e^q \E\l[[M,M]^{q/2}\r]\r)  \\
& =O\l(\frac{1}{n^q} + \e^q\r),  
\end{align*}
under Q2[$q$]. We used the Burkholder-Davis-Gundy inequality in the last inequality. This completes the proof. 
\end{proof}

Form \eqref{lem:3.1-2} and \eqref{X-X^0} in the above proof, the following corollary is obvious. 

\begin{cor}\label{cor:lem:3.1}
Suppose A1 and Q2[$q$] for some $q>0$. Then it holds that 
\begin{align}
\sup_{\e>0} \E\l[\l(\e^{-1} \|X-X^0\|_*\r)^q\r] < \infty. \label{cor:lem:3.1-1}
\end{align}
In addition, if $(n\e)^{-1} = O(1)$, then 
\begin{align}
\sup_{n\in \N,\e>0} \E\l[\l(\e^{-1} \|Y^{n,\e}-X^0\|_*\r)^q\r] < \infty. \label{cor:lem:3.1-2}
\end{align}
\end{cor}

\begin{lemma}\label{lem:3.2}
Under A1, it follows that 
\[
\lim_{m\to \infty} \tau_m^{n,\e}= \infty\quad  a.s., 
\] 
uniformly in $n\in \N$ and $\e\in [0,1]$. 
\end{lemma}
\begin{proof}
Noticing \eqref{sup-Q}, we have by the same argument as in the proof of Lemma 3.2 by Long {\it et al.} \cite{lss13} that 
\[
\sup_{n\in \N,\e>0}|Y_t^{n,\e}| \le \sqrt{2}\l(|x| + \sup_{s\in [0,t]}|Q_s| + t\r) e^{Ct^2} < \infty\quad a.s.
\]
for any $t>0$. Therefore we have the consequence. 
\end{proof}

\begin{lemma}\label{lem:3.3}
Let $g\in C^{1,1}_\uparrow(\R^d\times\Theta;\R)$. Suppose A1, A2, Q1[$\g$], and that 
\[
\d_{n,\e}\D_n^{-1}\to \infty,\quad \e\,\D_n^{\g}\,\d_{n,\e}^{-1} =O(1),\quad n\e\to \infty. 
\]
Then, we have 
\begin{align}
\frac{1}{n}\sum_{k=1}^n g_{k-1}(\th)\filter  \toP \int_0^1 g(X_t^0,\th)\,\df t, \label{lem:3.3-1}
\end{align}
uniformly in $\th\in\Theta$. 
In addition, suppose that Q2[$q$] holds for some $q>0$, then 
\begin{align}
\sup_{n\in\N,\e>0}\E\l[\l(\e^{-1}\sup_{\th\in \Theta}\l|\frac{1}{n}\sum_{k=1}^n g_{k-1}(\th)\filter - \int_0^1 g(X_t^0,\th)\,\df t\r|\r)^q\r]<\infty.  \label{lem:3.3-2}
\end{align}
\end{lemma}

\begin{proof}
Since $|g(x)|\lesssim (1 + |x|)^C$, we have that 
\begin{align*}
\sup_{\th\in\Theta}&\l|\frac{1}{n}\sum_{k=1}^n g_{k-1}(\th)\filter  - \int_0^1 g(X_t^0,\th)\,\df t\r| \\
&\lesssim \sup_{\th\in\Theta}\l|\frac{1}{n}\sum_{k=1}^n g_{k-1}(\th)  - \int_0^1 g(X_t^0,\th)\,\df t\r| 
+ \frac{1}{n}\sum_{k=1}^n (1 + |X_\t|)^C \I_{\{|\DX|>\d_{n,\e}\}} \\
&= \sup_{\th\in\Theta}\l|\int_0^1 \l\{g(Y_{t}^{n,\e})\,\ - g(X_t^0,\th)\r\}\,\df t\r| 
+ \frac{1}{n}\sum_{k=1}^n (1 + |X_\t|)^C \I_{\{|\DX|>\d_{n,\e}\}}. 
\end{align*}
For the first term on the right-hand side of the inequality, it holds that
\begin{align}
\sup_{\th\in\Theta}&\l|\int_0^1 \l\{g(Y_{t}^{n,\e})\,\ - g(X_t^0,\th)\r\}\,\df t\r| \\
&\le \sup_{\th\in \Theta}\int_0^1 \int_0^1 \l|\n_x g(X_s^0 + u(Y_s^{n,\e}- X_s^0),\th)\r|\cdot\l|Y_s^{n,\e}- X_s^0\r| \,\df u\df s \notag \\
&\lesssim \int_0^1 \l(1 + |X_s^0| + |Y_s^{n,\e}|^C\r)\l|Y_s^{n,\e}- X_s^0\r| \,\df s \notag \\
&\lesssim \l(1 + \sup_{t\in [0,1]}|X_t^0| + \sup_{t\in [0,1]}|X_t^0|\r)^C \sup_{t\in [0,1]}\l|Y_t^{n,\e} - X_t^0\r| \toP 0, \label{lem:3.3-3}
\end{align}
by Lemma \ref{lem:3.1}, \eqref{lem:3.1-1}. 
Hence the proof ends if we show the second term tends to zero in probability. 
Let 
\[
\xi_{k,n,\e}(\th):=\frac{1}{n}(1 + |X_\t|)^C\I_{\{|\DX|>\d_{n,\e}\}}. 
\]
We show that 
\begin{align}
\sum_{k=1}^n \E\l[\xi_{k,n,\e}(\th)|\F_\t\r]&\toP 0; \label{cond1}\\
\sum_{k=1}^n \E\l[|\xi_{k,n,\e}(\th)|^2|\F_\t\r]&\toP 0. \label{cond2}
\end{align}
which implies that $\sum_{k=1}^n \xi_{k,n,\e}(\th)\toP 0$ for each $\th\in\Theta$ from Lemma 9 by Genon-Catalot and Jacod \cite{gj93}. 
First, we show that \eqref{cond1}. Note that 
\begin{align*}
\sum_{k=1}^n \E\l[\xi_{k,n,\e}(\th)|\F_\t\r]&= \frac{1}{n}\sum_{k=1}^n (1 + |X_\t|)^C\P\l(|\DX|>\d_{n,\e}|\F_\t\r)
\end{align*}
Since it follows that $n^{-1}\sum_{k=1}^n (1 + |X_\t|)^C=O_p(1)$ by Lemma 3.3 in \cite{lss13}, it suffices to show that 
\[
\P\l(|\DX|>\d_{n,\e}|\F_\t\r) = o_p(1), 
\]
for any $k=1,\dots,n$. Note that it holds that 
\[
\sup_{t\in (\t,\T]}|X_t - X_\t| \lesssim \D_n(1 + \|X\|_*) + \e \sup_{t\in (\t,\T]}|Q_t^\e|. 
\]
Hence it follows from Q1[$\g$] and the assumption on $\d_{n,\e}$ that, for $n$ large enough, 
\begin{align}
\P\l(|\DX|>\d_{n,\e}|\F_\t\r) 
&\le \P\l(\sup_{t\in (\t,\T]}|Q_t-Q_\t| > \frac{\d_{n,\e}}{2\e} \bigg|\F_\t\r) \notag \\
&\quad +  \P\l((1 + \|X\|_*) > \frac{\d_{n,\e}}{2\D_n} \bigg|\F_\t\r) \notag \\ 
&\le \P\l(\sup_{t\in (\t,\T]}|Q_t-Q_\t| \gtrsim \D_n^\g \big|\F_\t\r) \notag \\
&\quad  +  \P\l((1 + \|X\|_*) \gtrsim \D_n^{-1}\d_{n,\e}\big|\F_\t\r) \notag \\ 
&= o_p(1), \label{Q1}
\end{align}
since $\|X\|_*$ is bounded in probability. 
This is the proof of \eqref{cond1}. The proof of \eqref{cond2} is similar to above, which ends the proof of \eqref{lem:3.3-1}. 
The proof of \eqref{lem:3.3-2} is easy from the estimates \eqref{lem:3.3-3} and Corollary \ref{cor:lem:3.1} since we are assuming that $n\e\to \infty$, 
so we omit the details. Then the proof is completed. 
\end{proof}

The next lemma is a version of the {\it Toeplitz lemma}; see also, e.g., Shimizu \cite{s12}. 
We need this result in the proof of the next Lemma \ref{lem:3.4}. 

\begin{lemma}\label{lem:t}
Let $\{a_k^n\}_{k=1}^n$ be a positive bounded sequence, and put $b_n:=\sum_{k=1}^n a_k^n$. Suppose that a sequence $\{x_k^n\}_{k=1}^n$ 
satisfies the following conditions: 
\begin{align}
&\sup_{n\in\N}|x_k^n|<\infty \quad \mbox{for each fixed }k; \label{t1}\\
&\lim_{m\to \infty}\sup_{k,n:\,m\le k\le n}|x_k^n - x|= 0\quad \mbox{for some $x\in\R$}; \label{t2}
\end{align}
Then, for any sequence $A_n$ with $A_n\sim b_n^{-1}$, $A_n\sum a_i^nx_i^n \to x$ as $n\to\infty$. 
\end{lemma}

\begin{lemma}\label{lem:3.4}
Let $g\in C^{1,1}_\uparrow(\R^d\times\Theta;\R^d)$. Under A1, it holds that 
\[
\sum_{k=1}^n g_{k-1}(\th)[\D_k^n Q^\e] \toP \int_0^1 g(X_t^0,\th)\,[\df Q_t]. 
\]
\end{lemma}

\begin{proof}
Since we are assuming the uniform convergence \eqref{Q}, we have that 
\begin{align*}
&\l|\sum_{k=1}^n g_{k-1}(\th)[\D_k^n Q^\e] - \int_0^1 g(X_t^0,\th)\,[\df Q_t]\r| \\
&=  \l| \sum_{k=1}^n g_{k-1}(\th)[\D_k^n Q^\e - \D_k^n Q]\r| + \l| \int_0^1 \l\{g(Y_t^{n,\e},\th) - g(X_t^0,\th)\r\}\,[\df Q_t]\r| \\
&\le \l|\sum_{k=1}^n g_{k-1}(\th)[\D_k^n Q^\e - \D_k^n Q]\r|
+ \int_0^1 \l|g(Y_t^{n,\e},\th) - g(X_t^0,\th)\r|\,[|\df A_t|]\\
&\quad  + \l|\int_0^1 \l\{g(Y_t^{n,\e},\th) - g(X_t^0,\th)\r\}\,[\df M_t]\r| \\
&=: I_{n,\e}^{(1)} + I_{n,\e}^{(2)} + I_{n,\e}^{(3)}. 
\end{align*}
As for $I_{n,\e}^{(1)}$, we suppose $d=1$ for simplicity of notation. Putting $a_k^n := g_k(\th) - g_{k-1}(\th)$ and $x_k^n:= Q^\e_{t_k} - Q_{t_k}$ (since $\e$ depends on $n$), we have that 
\begin{align*}
I_{n,\e}^{(1)} &= \sum_{k=1}^n \l\{ g_k(\th)[Q^\e_{t_k} - Q_{t_k}] - g_{k-1}(\th)[Q^\e_{t_{k-1}} - Q_{t_{k-1}}]\r\} \\
&\quad  - \sum_{k=1}^n \l\{g_k(\th) - g_{k-1}(\th)\r\}[Q^\e_{t_k} - Q_{t_k}] \\
&= g(X_1^\e)[Q^\e_1 - Q_1] - \sum_{k=1}^n a_k^n \cdot x_k^n.  
\end{align*}
Now, it is clear that $g(X_1^\e)[Q^\e_1 - Q_1]\to 0\ a.s.$ by the assumption \eqref{Q}. Moreover, the sum $\sum_{k=1}^n a_k^n \cdot x_k^n$ 
also converges to zero with probability one by using Lemma \ref{lem:t}. 
Indeed, the convergence \eqref{t1} is clear from \eqref{Q}, and that 
\[
\lim_{k\to \infty}\sup_{j,n:\,k\le j\le n}|x_j^n| \lesssim \lim_{\e\to 0}\|Q^\e-Q\|_* = 0\quad a.s.
\]
Moreover, as $\e\to 0$, 
\[
b_n=\sum_{k=1}^n a_k^n =\sum_{k=1}^n \l[g_k(\th) - g_{k-1}(\th)\r] \to g(X_1^0,\th) - g(X_0^0,\th) < \infty\quad a.s.\quad 
\]
Hence we see by the above Toeplitz lemma that $\sum_{k=1}^n a_k^n \cdot x_k^n\to 0$ with probability one. 
Therefore we have that $I_{n,\e}^{(1)}\to 0$ with probability one.

As for $I_{n,\e}^{(2)}$, it follows from Lemma \ref{lem:3.1}, \eqref{lem:3.1-1} that 
\begin{align*}
I_{n,\e}^{(2)} &\lesssim  (1 + \sup_{t\in [0,1]}|X_t^0| + \sup_{t\in [0,1]}|X_t|)^C \sup_{t\in [0,1]}\l|Y_t^{n,\e} - X_t^0\r|\toP 0. 
\end{align*}

As for $I_{n,\e}^{(3)}$, using Markov's, and the Burkholder-Davis-Gundy inequalities, we have for any $\eta>0$ that 
\begin{align*}
\P(|I_{n,\e}^{(3)}| > \eta) &\le \P(\tau_m^{n,\e} < 1) + \P\l(\l|\int_0^1 \l\{g(Y_t^{n,\e},\th) - g(X_t^0,\th)\r\}\I_{\{t\le \tau_m^{n,\e}\}}\,[\df M_t] \r|>\eta/2\r) \\
&\le \P(\tau_m^{n,\e} < 1) + 2\eta^{-1} \E\l|\int_0^1 \l\{g(Y_t^{n,\e},\th) - g(X_t^0,\th)\r\}\I_{\{t\le \tau_m^{n,\e}\}}\,[\df M_t]\r| \\
&\lesssim \P(\tau_m^{n,\e} < 1) + \eta^{-1} \E\l|\int_0^1 \l|g(Y_t^{n,\e},\th) - g(X_t^0,\th)\r|^2\I_{\{t\le \tau_m^{n,\e}\}}\,\df [M,M]_t\r|^{1/2}. 
\end{align*}
From the definition of $\tau_m^{n,\e}$, the integrand in the last term is bounded in $n$ and $\e$. 
Hence, taking the limit $n\to \infty$ and $\e\to 0$, we see from the dominated convergence theorem that $\P(|I_{n,\e}^{(3)}| > \eta)\to 0$, which completes the proof. 
\end{proof}

\begin{lemma}\label{lem:3.5}
Let $g\in C^{1,1}_\uparrow(\R^d\times\Theta;\R^d)$. Assume A1, A2, Q1[$\g$], and that 
\[
\d_{n,\e}\D_n^{-1}\to \infty,\quad \e\,\D_n^{\g}\,\d_{n,\e}^{-1} =O(1),\quad n\e\to \infty. 
\]
Then we have 
\begin{align}
\sum_{k=1}^n g_{k-1}(\th)[\chi_k(\th_0)] \filter  \toP 0, \label{lem:3.5-1}
\end{align}
uniformly in $\th\in\Theta$. In addition, assume Q2[$q$] for any $q>p=\dim (\Theta)$. Then 
\begin{align}
\E\l[\l(\e^{-1}\sup_{\th\in \Theta}\l|\sum_{k=1}^n g_{k-1}(\th)[\chi_k(\th_0)]  \filter\r|\r)^q\r]<\infty.  \label{lem:3.5-2}
\end{align}
\end{lemma}
\begin{proof}
The proof of \eqref{lem:3.5-1} is similar to the one of Lemma 3.5 by Long {\it et al.} \cite{lss13} with a slight extension to semimartingale version; see also Remark 4.3 in \cite{lss13}. 
It is clear from their proof that the indicator $\filter$ is not essential to the proof. So it is omitted. 

As for \eqref{lem:3.5-2}, we note that 
\begin{align*}
&\e^{-1}\sup_{\th\in \Theta}\l|\sum_{k=1}^n g_{k-1}(\th)[\chi_k(\th_0)]  \filter\r| \\
&=\e^{-1}\sup_{\th\in \Theta} \Bigg|\sum_{k=1}^n\int_0^1 g(Y_s^{n,\e},\th)\l[b(X_s,\th_0) - b(Y_s^{n,\e},\th_0) + \e\cdot \df Q^\e_t\r]\,\df s \cdot \filter \Bigg| \\
&\le \e^{-1}\int_0^1  \sup_{\th\in \Theta} \l|g(Y_s^{n,\e},\th)\l[b(X_s,\th_0) - b(Y_s^{n,\e},\th_0) \r]\r|\,\df s  \\
&\quad + \sup_{\th\in \Theta} \l|\int_0^1 g(Y_s^{n,\e},\th) [\df Q^\e_t - \df Q_t]\r| + \sup_{\th\in \Theta} \l|\int_0^1 g(Y_s^{n,\e},\th) [\df Q_t]\r| \\
&=: I^{(1)}_{n,\e} + I^{(2)}_{n,\e} + I^{(3)}_{n,\e}. 
\end{align*}
By the assumption A1 and the condition for $g$, we have 
\begin{align*}
I^{(1)}_{n,\e} &\lesssim \e^{-1} \int_0^1 (1 + |Y_s^{n,\e}|)^\la|X_s - Y_s^{n,\e}|\,ds \\
&\lesssim \l(1 + \|Y^{n,\e}- X^0\|_*^\la + \|X^0\|_*^\la \r) \l(\e^{-1}\|X - X^0\| + \e^{-1}\|X^0 - Y^{n,\e}\|\r)
\end{align*}
Hence, under the assumption that $n\e\to \infty$, Corollary \ref{cor:lem:3.1} yields that 
\begin{align*}
\E|I^{(1)}_{n,\e}|^q < \infty.  
\end{align*}
We have already shown that $I^{(2)}_{n,\e}\toP 0$ in the proof of Lemma \ref{lem:3.4}. 
Hence the proof ends if we show that $I^{(3)}_{n,\e}\toP 0$. 
Noticing that a bounded convex set $\Theta$ admits the following {\it Sobolev inequality}; 
\[
\sup_{\th\in \Theta}|u(\th)| \lesssim \|u(\th)\|_{L^q(\Theta)} + \|\n_\th u(\th)\|_{L^q(\Theta)}, 
\]
for $q>p=\dim (\Theta)$, we see for any $q>p$ that 
\begin{align*}
I^{(3)}_{n,\e} &\le \int_0^1 (1 + |Y_s^{n,\e},\th|)^\la \cdot |\df A_s| +  \sup_{\th\in \Theta} \l|\int_0^1 g(Y_s^{n,\e},\th) [\df M]\r| \\
&\lesssim \l(1 + \|Y^{n,\e}- X^0\|_*^\la + \|X^0\|_*^\la \r) TV(A) + \l(\int_\Theta \l|\int_0^1 g(Y_s^{n,\e},\th) [\df M^\e]\r|^q\,\df \th\r)^{1/q}. 
\end{align*}
Then, by using the Burkholder-Davis-Gundy inequality, we have 
\begin{align*}
\E|I^{(3)}_{n,\e}|^q &\lesssim \E\l[\l(1 + \|Y^{n,\e}- X^0\|_*^{\la q} \r) TV(A)^q \r] + \int_\Theta \E\l|\int_0^1 g(Y_s^{n,\e},\th) [\df M]\r|^q\,\df \th \\
&\lesssim 1 + \E\l[\l(1 + \|Y^{n,\e}- X^0\|_*^\la + \|X^0\|_*^\la \r)^q |[M,M]_1|^{q/2}\r] < \infty, 
\end{align*}
under Q2[$q$] for any $q>p$. This completes the proof of \eqref{lem:3.5-2}. 
\end{proof}

\begin{lemma}\label{lem:3.5-cor}
Let $g\in C^{1,1}_\uparrow(\R^d\times\Theta;\R^d)$. Assume A1, A2, Q1[$\g$], and that 
\begin{align*}
\d_{n,\e}\D_n^{-1}\to \infty,\quad \e\,\D_n^{\g}\,\d_{n,\e}^{-1} =O(1),\quad n\e\to \infty. 
\end{align*}
Then we have 
\begin{align*}
\sum_{k=1}^n g_{k-1}(\th)[\chi_k(\th_0)] \I_{\{|\DX|> \d_{n,\e}\}} = o_p(\e), 
\end{align*}
for each $\th\in \Theta$. 
\end{lemma}

\begin{proof}
Using $\DX=\int_\t^\T b(X_t,\th_0)\,\df t + \e \D_k^nQ^\e$, we have that 
\begin{align*}
\e^{-1} \sum_{k=1}^n &g_{k-1}(\th)[\chi_k(\th_0)] \I_{\{|\DX|> \d_{n,\e}\}} \\
&= \sum_{k=1}^n \e^{-1}\int_\t^\T g_{k-1}(\th)[b(X_t,\th_0) - b_{k-1}(\th_0)]\,\df t\cdot \I_{\{|\DX|> \d_{n,\e}\}} \\
&\quad + \sum_{k=1}^n g_{k-1}(\th)[\D_k^n Q^\e] \cdot \I_{\{|\DX|> \d_{n,\e}\}} \\
&=: H^{(1)}_{n,\e}(\th) + H^{(2)}_{n,\e}(\th). 
\end{align*}

As for $H^{(1)}_{n,\e}$, we easily find it converges to zero in probability as $n\e\to \infty$ 
by the same argument as for $H^{(1)}_{n,\e}$ in Lemma 3.6 by Long {\it et al.} \cite{lss13}. 

As for $H^{(2)}_{n,\e}$, it follows from Lemma \ref{lem:3.3} and \ref{lem:3.4} that 
\begin{align*}
H^{(2)}_{n,\e}(\th) &= \sum_{k=1}^n g_{k-1}(\th)[\D_k^n Q^\e] - \sum_{k=1}^n g_{k-1}(\th)[\D_k^n Q^\e]\filter \toP 0,
\end{align*}
for each $\th\in \Theta$.  Hence the proof is completed. 

\end{proof}

\subsection{Proof of Theorem \ref{thm:consist}}

We shall show that $\wh{\th}_{n,\e}$ is asymptotically equivalent to $\wh{\th}_{n,\e}^{LSE}$ given in \eqref{lse1}. 
Let 
\begin{align*}
\wt{\Psi}_{n,\e}(\th):=\e^2\l(\Psi_{n,\e}(\th) - \Psi_{n,\e}(\th_0)\r), \\
\wt{\Phi}_{n,\e}(\th):=\e^2\l(\Phi_{n,\e}(\th) - \Phi_{n,\e}(\th_0)\r). 
\end{align*}
where $\Psi_{n,\e}$ and $\Phi_{n,\e}$ are given in \eqref{lse2} and \eqref{tte2}. 
Then $\wh{\th}_{n,\e}^{LSE}$ and $\wh{\th}_{n,\e}$ are respectively minimum contrast estimators for contrast functions $\Psi_{n,\e}$ and $\Phi_{n,\e}$. 

By the same argument as in the proof of Theorem 2.1 with Remark 4.3 by Long {\it et al.} \cite{lss13}, 
all we need to show is 
\begin{align*}
\sup_{\th\in\Theta} |\wt{\Phi}_{n,\e}(\th) - F(\th_0)|\toP 0,  
\end{align*}
where $F(\th):= \int_0^1 |b(X_t^0,\th) - b(X_t^0,\th_0)|^2\,\df t$. 
Since $\sup_{\th\in\Theta}|\wt{\Psi}_{n,\e}(\th) - F(\th_0)|\toP 0$, 
Therefore, 
\begin{align*}
\sup_{\th\in\Theta} \l|\wt{\Phi}_{n,\e}(\th) - F(\th_0)\r|
&\le \sup_{\th\in\Theta} \l|\wt{\Phi}_{n,\e}(\th) - \wt{\Psi}_{n,\e}(\th)\r| + \sup_{\th\in\Theta} \l|\wt{\Psi}_{n,\e}(\th) - F(\th_0)\r| \\
&= \sup_{\th\in\Theta} \l|n\sum_{k=1}^n |\chi_{k}(\th)|^2\I_{\{|\DX|> \d_{n,\e}\}}\r| + o_p(1) \\
&\lesssim \sup_{\th\in\Theta} \l|\sum_{k=1}^n (b_{k-1}(\th) - b_{k-1}(\th_0))[\chi_{k}(\th_0)]\I_{\{|\DX|> \d_{n,\e}\}}\r| \\
&+ \sup_{\th\in\Theta} \l|\frac{1}{n}\sum_{k=1}^n |b_{k-1}(\th) - b_{k-1}(\th_0)|^2\I_{\{|\DX|> \d_{n,\e}\}}\r| + o_p(1) 
\end{align*}
The last first and second terms converges to zero in probability by Lemmas \ref{lem:3.3} and \ref{lem:3.5}. This completes the proof.

\subsection{Proof of Theorem \ref{thm:asym-dist}}

We use the following notation: 
\begin{itemize}
\item $G_{n,\e}(\th) = 2^{-1}\n_\th \Phi_{n,\e}(\th)\ \in \R^p$; 
\item $K_{n,\e}(\th) = \n_\th G_{n,\e}(\th) \,\l(=\n_\th^2 \Phi_{n,\e}(\th)\r)\ \in \R^p \otimes \R^p$; 
\item $K(\th)=\int_0^1 \n_\th^2b(X_t^0,\th)[b(X_t^0,\th_0) - b(X_t^0,\th)]\,\df t - I(\th_0)\ \in  \R^p \otimes \R^p$. 
\end{itemize}
Then it follows by Taylor's formula that, for some $\rho\in (0,1)$, 
\[
\int_0^1 K_{n,\e}\l(\th_0 + u(\wh{\th}_{n,\e} - \th_0)\r)\,\df u \cdot \e^{-1}\l(\wh{\th}_{n,\e} - \th_0\r) = \e^{-1}G_{n,\e}(\wh{\th}_{n,\e}) - \e^{-1}G_{n,\e}\l(\th_0\r). 
\]
Let us show that 
\begin{align}
&\e^{-1} G_{n,\e}(\th_0) \toP \int_0^1 \n_\th b(X_t^0,\th_0)\,[\df Q_t];\label{lem:3.6} \\
&\sup_{\th\in \Theta}\l|K_{n,\e}(\th) - K(\th)\r| \toP 0. \label{lem:3.7}
\end{align}
Then the result follows by the same argument as in the proof of Theorem 2.2 by Long {\it et al.} \cite{lss13}; see also Uchida \cite{u08}. 

As for \eqref{lem:3.6}: it follows that 
\begin{align*}
\e^{-1} G_{n,\e}(\th_0) &= \e^{-1}\sum_{k=1}^n \n_\th b_{k-1}(\th_0)[\chi_k(\th_0)] \filter \\
&= \e^{-1}\sum_{k=1}^n \n_\th b_{k-1}(\th_0)[\chi_k(\th_0)] - \e^{-1}\sum_{k=1}^n \n_\th b_{k-1}(\th_0)[\chi_k(\th_0)] \I_{\{|\DX|> \d_{n,\e}\}} \\
&= \e^{-1}\sum_{k=1}^n \n_\th b_{k-1}(\th_0)\l[\int_\t^\T \{b(X_t,\th_0) - b_{k-1}(\th_0)\}\,\df t\r] + \sum_{k=1}^n \n_\th b_{k-1}(\th_0)[\D_k^nQ] \\
&\quad - \e^{-1}\sum_{k=1}^n \n_\th b_{k-1}(\th_0)[\chi_k(\th_0)] \I_{\{|\DX|> \d_{n,\e}\}}
\end{align*}
Then we can show that the last first term converges to zero in probability by the same evaluation as for $H_{n,\e}^{(1)}(\th_0)$ in the proof of Lemma 3.6 in \cite{lss13}, 
the second term converges to $\int_0^1 \n_\th b(X_t^0,\th_0)[\df Q_t]$ in probability by Lemma \ref{lem:3.4}, and that 
the third term goes to zero in probability by Lemma \ref{lem:3.5-cor}. 
Similarly, as for \eqref{lem:3.7}, it follows that 
\begin{align*}
K_{n,\e}(\th) &= \sum_{k=1}^n \n_\th^2b_{k-1}(\th)[\chi_k(\th)] \filter - \frac{1}{n}\sum_{k=1}^n \n_\th b_{k-1}(\th)^\top \n_\th b_{k-1}(\th)\filter \\
&= \sum_{k=1}^n \n_\th^2b_{k-1}(\th)[\chi_k(\th_0)] \filter \\
&\quad+ \frac{1}{n}\sum_{k=1}^n \n_\th^2 b_{k-1}(\th)[b_{k-1}(\th_0) - b_{k-1}(\th)] \filter \\
&\quad- \frac{1}{n}\sum_{k=1}^n \n_\th b_{k-1}(\th)^\top \n_\th b_{k-1}(\th)\filter, 
\end{align*}
cf. the expression of $K_{n,\e}^{ij}(\th)$ in the proof of Lemma 3.7 by Long {\it et al.} \cite{lss13}. 
Hence Lemmas \ref{lem:3.3} and \ref{lem:3.5} yield \eqref{lem:3.7}. 
Using the facts \eqref{lem:3.6} and \eqref{lem:3.7}, and the consistency result: Theorem \ref{thm:consist}, 
we can show that 
\[
\e^{-1}\l(\wh{\th}_{n,\e} - \th_0\r) \sim^p -K^{-1}(\th_0)\cdot \e^{-1}G_{n,\e}(\th_0), \quad n\to \infty,\ \e\to 0, 
\]
by completely the same argument as in the proof of Theorem 2.2 by Long {\it et al.} \cite{lss13}. Therefore, the proof is completed. 

\subsection{Proof of Theorem \ref{thm:mighty}}

Denote by $\wh{u}:=\e^{-1}(\wh{\th}_{n,\e} - \th_0)$. 
Since $\wh{u} \toP \z$ from Theorem \ref{thm:asym-dist}, 
the proof ends if we show that $\wh{u}$ is $L^p$-bounded: $\sup_{n,\e}\E|\wh{u}|^p <\infty$ for any $p>0$. 
For this proof, let $U_{n,\e}(\th_0) = \l\{u\in \R^p\,:\,\th_0 + \e u \in \Theta_0\r\}$, and define random fields $\Z_{n,\e}:U_{n,\e}(\th_0)\to \R_+$ by 
\begin{align*}
\Z_{n,\e}(u) &= \exp\l\{ -  \Phi_{n,\e}(\th_0 + \e u) +  \Phi_{n,\e}(\th_0)\r\}, \quad u\in U_{n,\e}(\th_0). 
\end{align*}
Then, since $\th_0\in \Theta_0$, we see that 
\begin{align}
\Z_{n,\e}(\wh{u}) \ge \sup_{u\in U_{n,\e}(\th_0)} \Z_{n,\e}(u) \ge \Z_{n,\e}(0)=1. \label{Z}
\end{align}
Setting $V_{n,\e}(r) := U_{n,\e}(\th_0) \cap \l\{u\in \R^p\,:\,|u|\ge r\r\}$, 
we consider the following condition for the random fields $\Z_{n,\e}$: for every $L>0$ and $r>0$, 
\begin{align}
\P\l(\sup_{u\in V_{n,\e}(r)} \Z_{n,\e}(u) \ge e^{-r} \r) \lesssim r^{-L}, \label{PLDI}
\end{align}
which is called the {\it polynomial type large deviation inequality (PLDI)}, and is investigated by Yoshida \cite{y11} in details. 
If this PLDI holds true then, for any $L>p$,  
\begin{align*}
\sup_{n\in \N,\e>0}\E|\wh{u}|^p &= \sup_{n\in \N,\e>0} p\int_0^\infty r^{p-1}\P\l(|\wh{u}| \ge r\r) \,\df r \\
&\le \sup_{n\in \N,\e>0} p\int_0^\infty r^{p-1}\l\{1 \wedge \P\l(\sup_{u\in V_{n,\e}(r)} \Z_{n,\e}(u) \ge 1 \r) \r\}\,\df r  \\
&\lesssim \int_0^\infty r^{p-1}(1\wedge r^{-L})\,\df r < \infty, 
\end{align*}
here we used \eqref{Z} in the first inequality. 
Therefore the proof ends if we show \eqref{PLDI}, some sufficient conditions for which are found in the paper by Yoshida \cite{y11}. 
Here we shall verify the conditions [A1$''$], [A4$'$], [A6], [B1] and [B2] given in Theorem 3, (c) in \cite{y11}. 
See also Ogihara and Yoshida \cite{oy11} or Masuda \cite{ma13} for simplified descriptions for those conditions. 

Applying Taylor's formula with the notation $G_{n,\e}(\th)$, $K_{n,\e}(\th)$ and $K(\th)$ given in the proof of Theorem \ref{thm:asym-dist}, we have 
\begin{align*}
\log \Z_{n,\e}(u) &= -  \Phi_{n,\e}(\th_0 + \e u) +   \Phi_{n,\e}(\th_0) \\
&=  -\e G_{n,\e}(\th_0)[u] - \frac{\e^2}{2}\l\{-K(\th_0)\r\}[u^{\otimes 2}] + R_{n,\e}(u), 
\end{align*}
where 
\begin{align*}
R_{n,\e}(u) &= \e^2 \int_0^1 (s-1) \l\{K(\th_0)[u^{\otimes 2}] - K_{n,\e}(\th_0 + s\cdot \e u)[u^{\otimes 2}]\r\}\,\df s \\
&= \frac{\e^2}{2}\l\{K_{n,\e}(\th_0) - K(\th_0)\r\}[u^{\otimes 2}] - \e^3 \int_0^1 (s-1) \int_0^1 \n_\th K_{n,\e}(\th_0 + ts\cdot\e u)[u^{\otimes 3}]\,\df t\df s. 
\end{align*}
This means that $\Z_{n,\e}$ could be {\it Partially Locally Asymptotically Quadratic (PLAQ)}, which is a starting point of \cite{y11}. 
According to Theorem 3, (c) in \cite{y11}, if we take some ``tuning parameters" given in [A4$'$] in \cite{y11} such as $\b_1\approx 1/2$, $\rho_1,\rho_2,\b,\b_2\approx 0$, 
then the PLDI \eqref{PLDI} holds true if the following [A1$''$], [A6], [B1] and [B2] are satisfied; we use the same conditioning numbers as in \cite{y11} to make those correspondences clear. 
\begin{itemize}
\item[{[A1$''$]}] For every $q>0$, 
\begin{align}
\sup_{n\in \N,\e>0} \E\l[\l(\e^2 \sup_{\th\in \Theta}\l|\n_\th^3 \Phi_{n,\e}(\th)\r|\r)^q\r] < \infty. \label{a1-1}
\end{align}
Moreover, for given $L>0$ and any $\d>0$ small enough, 
\begin{align}
\sup_{n\in \N,\e>0} \E\l[\l(\e^{-1} \l|K_{n,\e}(\th_0) - K(\th_0)\r|\r)^{L-\d}\r] < \infty. \label{a1-2}
\end{align}

\item [{[A6]}] For any $\d>0$ small enough, 
\begin{align}
&\sup_{n\in \N,\e>0} \E\l[\l|\e G_{n,\e}(\th_0)\r|^{L+\d}\r] < \infty; \label{a6-1} \\
&\sup_{n\in \N,\e>0} \E\l[\sup_{\th\in \Theta}\l(\e^{-1} \l|\wt{\Phi}_{n,\e}(\th) - F(\th)\r|\r)^{L+\d}\r] < \infty,  \label{a6-2}
\end{align}
where $\wt{\Phi}_{n,\e}$ and $F(\th) = \int_0^1 |b(X_t^0,\th) - b(X_t^0,\th_0)|^2\,\df t$ are given in the proof of Theorem \ref{thm:consist}. 

\item[{[B1]}] The matrix $-K(\th_0)\ (=I(\th_0))$ is deterministic and positive definite. 

\item[{[B2]}] There exists a deterministic positive number $\chi$ such that 
\[
-\l\{F(\th) - F(\th_0)\r\} \le -\chi|\th - \th_0|^2. 
\]
\end{itemize}
(Note that the notational correspondence between \cite{y11} and ours is: $a_T=\e$; $b_T=\e^{-2}$; $\H_T=-\Phi_{n,\e}$; $\Y=-F$ and $\G=-K$). 

Now we can easily check that the conditions \eqref{a1-1} and \eqref{a6-1} are true by Lemma \ref{lem:3.5}, \eqref{lem:3.5-2}, and that 
\eqref{a1-2} and \eqref{a6-2} are also true by Lemma \ref{lem:3.3}, \eqref{lem:3.3-2}. 
Moreover the conditions [B1] and [B2] are clear from the assumptions A1 and A4, respectively, 
the proof ends if we show \eqref{a1-1}--\eqref{a6-2}. 
Hence the proof is completed.

\subsection{Proof of Theorem \ref{thm:finite-activity}}

First, we shall show the consistency: 
\[
\wt{\th}_{n,\e,\d} \toP \th_0. 
\]
Since we suppose that the jumps are specified, the following ``negligibility" is obtained: 
\begin{align}
\P\l(\|\D Q^\e\|^*_k > \d\r) &= 1 - \P\l\{\sup_{t\in (\t,\T]}|\D Q_t| \le \d/\e\r\} \notag \\
&= 1 - e^{ - \la(\d/\e) \D_n} \to 0 \label{negligible2}
\end{align}
since $\la(\d/\e) \D_n\to 0$. 
Then we have the following lemma, which is the same type of results as Lemmas \ref{lem:3.3} and \ref{lem:3.5}. 

\begin{lemma}\label{lem3.5-6}
Let $g\in C^{1,1}_\uparrow(\R^d\times\Theta;\R)$. Suppose A1, A2, and that 
\[
\d/\e\to 0. 
\]
Then, we have 
\begin{align*}
&\frac{1}{n}\sum_{k=1}^n g_{k-1}(\th) \qfilter \toP \int_0^1 g(X_t^0,\th)\,\df t,\\
&\sum_{k=1}^n g_{k-1}(\th)[\chi_k(\th_0)] \qfilter \toP 0, 
\end{align*}
uniformly in $\th\in\Theta$. 
\end{lemma}
The proof of this lemma is an obvious modification of the proofs of Lemmas \ref{lem:3.3} and \ref{lem:3.5} 
using the negligibility condition \eqref{negligible2}. 
Then by the same argument as in the proof of Theorem \ref{thm:consist}, the consistency follows. 

Next, note that, as in the proof of Theorem \ref{thm:asym-dist}, 
\[
\int_0^1 K_{n,\e,\d}\l(\th_0 + u(\wt{\th}_{n,\e,\d} - \th_0)\r)\,\df u \cdot \l(\wt{\th}_{n,\e,\d} - \th_0\r) = G_{n,\e,\d}(\wh{\th}_{n,\e}) - G_{n,\e,\d}\l(\th_0\r). 
\]
where 
\begin{align*}
G_{n,\e,\d}(\th) &= 2^{-1}\n_\th \wt{\Phi}_{n,\e,\d}(\th); \quad K_{n,\e,\d}(\th) = \n_\th G_{n,\e,\d}(\th); \\
K(\th)&=\int_0^1 \n_\th^2b(X_t^0,\th)[b(X_t^0,\th_0) - b(X_t^0,\th)]\,\df t - I(\th_0). 
\end{align*}
If we show that 
\begin{align}
\e^{-1}G_{n,\e,\d}\l(\th_0\r) \toP \int_0^1 \n_\th b(X_t^0,\th_0)\l[\df W_t\r], \label{asy-normal0}
\end{align}
then we obtain the consequence because the convergence 
\begin{align*}
\sup_{\th\in \Theta}\l|K_{n,\e,\d}(\th) - K(\th)\r| \toP 0
\end{align*}
holds true due to Lemma \ref{lem3.5-6} and the same argument as in the proof of Lemma 3.7 in \cite{lss13}. 

Note that 
\begin{align*}
\e^{-1}G_{n,\e,\d}\l(\th_0\r) 
&= \e^{-1} \sum_{k=1}^n \n_\th b_{k-1}(\th_0)\l[\int_\t^\T\{b(X_s,\th_0) - b_{k-1}(\th_0)\}\,\df t\r]\I_{\l\{\|\D Q^\e\|^*_k \le \d\r\}} \\
&\quad + \sum_{k=1}^n \n_\th  b_{k-1}(\th_0)\l[\sum_{i=N_\t+1}^{N_\T} Y_i\I_{\{|Y_i|\le \d\}}\r]  \\
&\quad + \sum_{k=1}^n \n_\th  b_{k-1}(\th_0)\l[\D_k^n W\r] \qfilter \\
&=: I^{(1)}_{n,\e,\d} + I^{(2)}_{n,\e,\d} + I^{(3)}_{n,\e,\d}.  
\end{align*}
Note that $I^{(1)}_{n,\e,\d} \to 0$ by the same argument as in the proof of Lemma 3.6, $H^{(1)}_{n,\e}(\th_0)$ in \cite{lss13} 
by the assumption \eqref{ne}. 
Moreover, it is easy to see that 
\begin{align*}
\E\l|I^{(2)}_{n,\e,\d}\r| = \frac{\la \d}{n\e}\sum_{k=1}^n \E|\n_\th  b_{k-1}(\th_0)| \to 0. 
\end{align*}
Furthermore, note that 
\[
I^{(3)}_{n,\e,\d} = \sum_{k=1}^n \n_\th  b_{k-1}(\th_0)\l[\D_k^n W\r] - \sum_{k=1}^n \n_\th  b_{k-1}(\th_0)\l[\D_k^n W\r] \I_{\l\{\|\D Q^\e\|^*_k > \d\r\}}
\]
Now, observe a measurability that 
\[
\l\{\omega\in \Omega\,:\, \|\D Q\|^*_k \le \d\r\} \in \s\Big(N\big((\t,s],[-\d,\d]\big)\,:\,s\in (\t,\T]\Big), 
\]
which is independent of $\D_k^nW$ and $X_\t$. Therefore, 
\begin{align*}
&\E\l| \sum_{k=1}^n \n_\th  b_{k-1}(\th_0)\l[\D_k^n W\r] \I_{\l\{\|\D Q^\e\|^*_k > \d\r\}}\r| \\
&\le  \sum_{k=1}^n \E\l| \n_\th b_{k-1}(\th_0)\l[\D_k^n W\r]\r| \P\l(\|\D Q^\e\|^*_k > \d\r) \\
&\le \sum_{k=1}^n \l(\E\int_{\t}^{\T}  trace \l(\n_\th b_{k-1}^{\otimes 2}(\th_0)\r)\,\df t\r)^{1/2}\l(1 - e^{-\la(\d/\e) \D_n}\r) \to 0
\end{align*}
from \eqref{negligible2}. 
As a result, Lemma \eqref{lem:3.4} yields that 
\[
I^{(3)}_{n,\e,\d} \toP  \int_0^1 \n_\th b(X_t^0,\th_0)\l[\df W_t\r]. 
\]
This completes the proof of \eqref{asy-normal0}.

\subsection{Proof of Theorem \ref{thm:infinite-activity}}

Note that, under the asymptotic conditions, it follows that 
\begin{align}
\P\l(\|\D Q^\e\|^*_k \le \d\r) &= e^{ - \la(\d/\e) \D_n} \to 0;\quad 
\P\l(\|\D Q^\e\|^*_k > \d\r)  \to 1, \label{poisson}
\end{align}
which is a different situation in the previous theorem. 

Under this setting, we can show the following lemma. 
\begin{lemma}\label{lem3.5-6-alt}
Let $g\in C^{1,1}_\uparrow(\R^d\times\Theta;\R)$. Suppose A1, A2, Q1[$\g$], and that 
\[
\frac{\la(\d/\e)}{n\log n} \to c \in (0,1),\quad n\e\D_n^\g\to 0.  
\]
Then it follows that, for $\eta_n:= e^{ - \la(\d/\e) \D_n}$,   
\begin{align}
&\frac{1}{n\eta_n}\sum_{k=1}^n g_{k-1}(\th) \qfilter \toP \int_0^1 g(X_t^0,\th)\,\df t, \label{lem8-1}\\
&\frac{1}{\eta_n}\sum_{k=1}^n g_{k-1}(\th)[\chi_k(\th_0)] \qfilter \toP 0,  \label{lem8-2}
\end{align}
uniformly in $\th\in\Theta$. 
\end{lemma}

\begin{flushleft}
{\it Proof of Lemma \ref{lem3.5-6-alt}.}
\end{flushleft}
As for \eqref{lem8-1}: note that 
\begin{align*}
&\sup_{\th\in \Theta}\l|\frac{1}{n\eta_n}\sum_{k=1}^n g_{k-1}(\th) \qfilter \toP \int_0^1 g(X_t^0,\th)\,\df t\r| \\
&\qquad \le \sup_{\th\in \Theta}\l|\frac{1}{n\eta_n}\sum_{k=1}^n g_{k-1}(\th) \qfilter  - \frac{1}{n}\sum_{k=1}^n g_{k-1}(\th)\r| \\
&\qquad + \sup_{\th\in \Theta}\l|\frac{1}{n}\sum_{k=1}^n g_{k-1}(\th) - \int_0^1 g(X_t^0,\th)\,\df t\r|. 
\end{align*}
and the second term on the right-hand side of the inequality converges to zero in probability from Lemma 3.3 by Long {\it et al.} \cite{lss13}. 
The last first term is rewritten as $\sum_{k=1}^n \xi_k^n(\th)$ with 
\[
\xi_k^n(\th) = \frac{1}{n} g_{k-1}(\th) \l(\frac{1}{\eta_n}\qfilter - 1\r). 
\]
Then we immediately see that, since $\F_\t$ and $\|\D Q^\e\|_k$ are independent each other, 
\[
\sum_{k=1}^n \E[\xi_k^n(\th)|\F_\t] = \frac{1}{n} \sum_{k=1}^n g_{k-1}(\th) \l[\eta_n^{-1} \P\l(\|\D Q^\e\|_k^* \le \d\r) - 1\r]= 0. 
\]
Moreover 
\begin{align*}
\sum_{k=1}^n \E[|\xi_k^n(\th)|^2|\F_\t] &= \frac{1}{n^2}\sum_{k=1}^n g_{k-1}^2(\th) \E\l|\frac{1}{\eta_n}\qfilter - 1\r|^2 \\
&= O_p\l(\frac{1}{n\eta_n}\r) = O_p\l(\frac{1}{e^{(1-c_2)\log n}}\r) \toP 0, 
\end{align*}
by the assumption. Hence we have that $\sum_{k=1}^n \xi_k^n(\th)\toP 0$ for every $\th\in \Theta$ 
from Lemma 9 by Genon-Catalot and Jacod \cite{gj93}. 

To prove the uniformity of convergence, we have to show the tightness of the sequence $\{\sum_{k=1}^n \xi_k^n(\cdot)\}_n$. 
We shall use Theorem 20 in Appendix 1 by Ibragimov and Has'minskii \cite{ih81}, that is, we shall show that, for some $H>0$ and any $N\in \N$, 
\begin{align}
\E\l|\sum_{k=1}^n \xi_k^n(\th)\r|^{2N} &< H; \label{tight1}\\
\E\l|\sum_{k=1}^n \l[\xi_k^n(\th_1) - \xi_k^n(\th_2)\r] \r|^{2N} &\le H|\th_1-\th_2|^{2N} \label{tight2}
\end{align} 

As for \eqref{tight1}: using the independent property of $X_\t$ and $\|\D Q^\e\|_k$, we have that 
\begin{align*}
\E\l|\sum_{k=1}^n \xi_k^n(\th)\r|^{2N} 
&=  \sum_{k=1}^n \E\l[|g_{k-1}(\th)|^{2N}\E\l[\l|\frac{1}{\eta_n}\qfilter - 1\r|^{2N}|\F_\t\r]\r] \\
&\lesssim \E\l[\sup_{\th\in \Theta, t\in [0,1]}|g(X_t)|^{2N}\r]\frac{1}{n^{2N}}\sum_{k=1}^n \l(\frac{1}{\eta_n^{2N-1}} +  1\r) \\
&= O\l(\frac{1}{(n\eta_n)^{2N-1}}\r) = O\l(n^{-(1-c)}\r) \to 0. 
\end{align*}
Therefore it is bounded. Inequality \eqref{tight2} is similarly proved since $g\in C^{1,1}_\uparrow(\R^d\times\Theta;\R)$.  
Hence \eqref{lem8-1} is proved. 

Finally we shall show \eqref{lem8-2}. Note that 
\begin{align*}
\sup_{\th\in \Theta}&\l|\frac{1}{\eta_n}\sum_{k=1}^n g_{k-1}(\th)[\chi_k(\th_0)] \qfilter\r| \\
&\le \frac{1}{\eta_n}\sum_{k=1}^n \int_\t^{\T} \sup_{\th\in \Theta}\l|g_{k-1}(\th) \l[b(X_s,\th_0) - b_{k-1}(\th_0)\r]\r|\,\df s\cdot \qfilter \\
&\quad + \frac{\e}{\eta_n} \sum_{k=1}^n \sup_{\th\in \Theta}\l|g_{k-1}(\th)\l[\D_k^nQ\r]\r|\qfilter \\
&=: J^{(1)}_n  + J^{(2)}_n. 
\end{align*}

As for $J^{(1)}_n$: it follows from the assumption A1 that, for $Y^{n,\e}$ given in Lemma \ref{lem:3.1}, 
\begin{align*}
|J^{(1)}_n| &= \frac{1}{\eta_n}\sum_{k=1}^n\int_\t^{\T}  \sup_{\th\in \Theta}\l| g(Y_s^{n,\e})\l[b(X_s,\th_0) - b(Y_s^{n,\e},\th_0)\r] \r|\,\df s\cdot \qfilter \\
&\lesssim \l(1 + \sup_{t\in [0,1]}|X_t|\r)^C \cdot \frac{1}{n\eta_n}\sum_{k=1}^n \sup_{t\in (\t,\T]}|X_t - Y_t^{n,\e}| \qfilter \\
&\lesssim \l(1 + \sup_{t\in [0,1]}|X_t|\r)^C \cdot \l(\|X - X^0\|_* + \|X^0 - Y^{n,\e}\|_*\r)\cdot \frac{1}{n\eta_n}\sum_{k=1}^n \qfilter
\end{align*} 
Since
\begin{align}
\E\l|\frac{1}{n\eta_n}\sum_{k=1}^n \qfilter\r| = 1 \quad \Rightarrow \quad \frac{1}{n\eta_n}\sum_{k=1}^n \qfilter=O_p(1), \label{op1}
\end{align}
and Lemma \ref{lem:3.1} we have that 
\[
|J^{(1)}_n|\toP 0. 
\]

As for $J^{(2)}_n$: it follows for some $C>0$ that  
\begin{align*}
|J^{(2)}_n| &\lesssim  \sup_{\th\in \Theta,t\in [0,1]}(1 +|X_t|)^C \cdot \frac{n\e}{\eta_n}\cdot 
\sum_{k=1}^n\frac{1}{n}\sup_{s\in (\t,\T]}|Q_s - Q_\t| \qfilter \\
&\le O_p\l(n\e\sup_{s\in (0,\D_n]}|Q_s|\r), 
\end{align*}
since $\sup_{s\in (\t,\T]}|Q_s - Q_\t|\sim^d \sup_{s\in (0,\D_n]}|Q_s|$ and \eqref{op1}. 
Moreover, noticing under Q1[$\g$] that 
\[
\sup_{s\in (0,\D_n]}|Q_s| = o_p(\D_n^\g), 
\]
we obtain that  
\[
J^{(2)}_n = O_p(n\e\D_n^\g) \to 0. 
\] 
This completes the proof of \eqref{lem8-2}. $\qquad \Box$

\ \vspace{5mm}\\
Now, putting 
\[
\ol{\Psi}_{n,\e}(\th):=\e^2\eta_n^{-1}\l(\wt{\Psi}_{n,\e,\d}(\th) - \wt{\Psi}_{n,\e,\d}(\th_0)\r); \quad 
F(\th):= \int_0^1 |b(X_t^0,\th) - b(X_t^0,\th_0)|^2\,\df t, 
\]
and using Lemma \ref{lem3.5-6-alt}, we can easily show that 
\[
\sup_{\th\in \Theta}\l|\ol{\Psi}_{n,\e}(\th) - F(\th)\r| \toP 0. 
\]
This and the identifiability condition A3 yield the consistency: $\wt{\th}_{n,\e,\d} \toP \th_0$. 

Suppose that the L\'evy process $Q$ is of the form 
\[
Q_t = a t +   c W_t + \int_0^t \int_{|z|\le 1} z\,\wt{N}(\df t,\df z)  + \int_0^t \int_{|z|> 1} z\,N(\df t,\df z), 
\]
where $a\in \R^d$, $c\ge 0$, $W$ is a $d$-dimensional Wiener process, $N$ is a Poisson random measure associated with jumps of $Q$, and 
$\wt{N}(\df t,\df z)= N(\df t,\df z) - \nu(\df z)\df t$, and note that, for any $\d>0$, 
\[
Q_t = \a_{\d/\e} t + W_t + \int_0^t \int_{|z|\le \d/\e} z\,\wt{N}(\df t,\df z)  + \int_0^t \int_{|z|> \d/\e} z\,N(\df t,\df z), 
\]
where $a_{\d/\e} = a - \int_{\d/\e < |z|\le 1}z\,\nu(\df z)$. 

Hereafter, we put 
\[
\zeta:=\d/\e.
\] 
Using the same notation as in the previous theorem, the proof ends if we show that 
\begin{align}
\l(\s(\zeta)\e\r)^{-1}G_{n,\e,\d}\l(\th_0\r) \toD \int_0^1 \n_\th b(X_t^0,\th_0)\l[\df B_t\r]. \label{asy-normal}
\end{align}
Note that 
\begin{align*}
\l(\s(\zeta)\e\r)^{-1}G_{n,\e,\d}\l(\th_0\r) 
&= \l(\s(\zeta)\e\r)^{-1}\sum_{k=1}^n \n_\th b_{k-1}(\th_0)\l[\int_\t^\T\{b(X_s,\th_0) - b_{k-1}(\th_0)\}\,\df t\r]\I_{\l\{\|\D Q^\e\|^*_k \le \d\r\}} \\
&\quad + \s^{-1}(\zeta) \sum_{k=1}^n \n_\th  b_{k-1}(\th_0)\big[a_\d \D_n + \D_k^n W\big] \I_{\l\{\|\D Q^\e \|^*_k \le \d\r\}} \\
&\quad + \s^{-1}(\zeta) \sum_{k=1}^n \n_\th  b_{k-1}(\th_0)\l[\int_\t^\T \int_{|z|\le \zeta} z\,\wt{N}(\df t,\df z)\r] \\
&=: I^{(1)}_{n,\e,\d} + I^{(2)}_{n,\e,\d} + I^{(3)}_{n,\e,\d}. 
\end{align*}

As for $I^{(2)}_{n,\e,\d}$: Since $\qfilter$ is independent of $\D_k^nW$ and $X_\t$, we have that 
\begin{align*}
\E\l|I^{(2)}_{n,\e,\d}\r| 
&\le \s^{-1}(\zeta) \frac{1}{n}\sum_{k=1}^n  \E\l|\n_\th  b_{k-1}(\th_0)a_\d \r| \P\l(\|\D Q\|^*_k \le \zeta\r) \\
&\quad  + \s^{-1}(\zeta) \sum_{k=1}^n \E\l|\n_\th  b_{k-1}(\th_0)\big[\D_k^n W\big]\r| \P\l(\|\D Q\|^*_k \le \d\r)  \\
& \le \s^{-1}(\zeta) \sum_{k=1}^n \l(\E\int_{\t}^{\T}  trace \l(\n_\th b_{k-1}^{\otimes 2}(\th_0)\r)\,\df t\r)^{1/2}\P\l(\|\D Q\|^*_k \le \zeta\r) \\
&= O\l(\s^{-1}(\zeta)e^{-\la(\zeta)\D_n}\r). 
\end{align*}
The last equality is due to \eqref{poisson}. Therefore we have 
\begin{align}
\E\l|I^{(2)}_{n,\e,\d}\r| &= O\l(\l(\s(\zeta) \sum_{k=0}^{\infty} \frac{(\D_n\la(\zeta))^k}{k!}  \r)^{-1}\r). \label{I2-1}
\end{align}
For an integer $M$ such that $1/M \le \rho$, we see that, 
\begin{align}
\s(\zeta) (\D_n\la(\zeta))^M = \l(\s^{1/M}(\zeta) \D_n\la(\zeta)\r)^M \ge \l(\s^{\rho}(\zeta) \D_n\la(\zeta)\r)^M \to \infty \label{I2-2}
\end{align}
due to the the condition \eqref{with-rho}. Hence we have 
\[
\E\l|I^{(2)}_{n,\e,\d}\r| \to 0. 
\]

As for $I^{(1)}_{n,\e,\d}$: By the same argument as in the proof of Lemma 3.6, $H^{(1)}_{n,\e}(\th_0)$ in \cite{lss13}, 
we can obtain the following inequality: 
\begin{align*}
|I^{(1)}_{n,\e,\d}| &\lesssim  \frac{1}{n\e \s(\zeta)}\frac{1}{n}\sum_{i=1}^n |\n_\th b_{k-1}(\th_0)|\cdot |b_{k-1}(\th_0)| \\
&\qquad + \frac{1}{n\s(\zeta)}\sum_{i=1}^n |\n_\th b_{k-1}(\th_0)| \sup_{s \in (\t,\T]}|Q_t - Q_{\t}|\qfilter 
\end{align*}
Using the same estimates for $J^{(2)}_n$ in the proof of Lemma \ref{lem3.5-6-alt}, 
we have that 
\begin{align*}
|I^{(1)}_{n,\e,\d}| =  O_p\l(\frac{1}{n\e \s(\zeta)}\r) +  o_p\l(\s^{-1}(\zeta) e^{-\la(\zeta)\D_n}\r), 
\end{align*}
Then, by the same estimates as for \eqref{I2-1} and \eqref{I2-2} above, we see that $|I^{(1)}_{n,\e,\d}| \toP 0$. 

As for $I^{(3)}_{n,\e,\d}$: Let 
\[
L_t^{\d} := \s^{-1}(\zeta) \int_0^t \int_{|z|\le \d} z\,\wt{N}(\df t,\df z).  
\]
Thanks to Theorem 2 by Asmussen and Rosinski \cite{ar01}, it follows under the assumption \eqref{sigma(delta)-1} that 
there exists a Wiener process $B$, independent of $W$, such that 
\[
L^{\zeta}\toD B\quad \mbox{in $\mathbb{D}[0,1]$}, 
\]
as $\d\to 0$. Moreover, since $\|Y^{n,\e}- X^0\|^*$ a.s. by Lemma \eqref{lem:3.1}, we have a joint convergence 
\[
\l(L^{\zeta},Y^{n,\e}\r) \toD (B,X^0) \quad \mbox{in $\mathbb{D}[0,1]$}. 
\]
as $\e,\d\to 0$ and $n\to \infty$. 
Hence it follows from Theorem 5.16 by Jacod and Shiryaev \cite{js03} that 
\begin{align*}
I^{(3)}_{n,\e,\d} &= \sum_{k=1}^n \n_\th  b_{k-1}(\th_0)\l[\D_k^n L^{\zeta}\r] = \int_0^1 \n_\th b\l(Y^{n,\e}_t,\th_0\r)\,\df L^\zeta_t \\
&\toD \int_0^1 \n_\th b(X^0_t,\th_0)\,\df B_t
\end{align*}
This completes the proof of \eqref{asy-normal}, and the statement is proved. 

\ \vspace{2mm}\\
\begin{flushleft}
{\large {\bf Acknowledgement}}
\end{flushleft}
The author would like to thank the anonymous referees for their valuable suggestions and proposals that significantly improve this manuscript. 
This research has been partially supported by JSPS KAKENHI Grant-in-Aid for Scientific Research (C), Grant Number JP15K05009. 


\end{document}